\documentclass[11pt,a4paper]{article}
\usepackage[T1]{fontenc}
\usepackage{lmodern,microtype}
\usepackage[margin=24mm]{geometry}
\usepackage{amsmath,amssymb,amsthm,mathtools}
\usepackage{booktabs,longtable,tabularx,array,graphicx}
\usepackage{xcolor,enumitem,fancyhdr,needspace,caption}
\usepackage[hidelinks]{hyperref}
\definecolor{navy}{HTML}{173B66}
\newcommand{\ec}{\eta_{\mathrm c}}
\newcommand{\dd}{\,\mathrm d}

\newcommand{\CC}{\mathbb C}
\newtheorem{theorem}{Theorem}[section]
\newtheorem{lemma}[theorem]{Lemma}
\newtheorem{proposition}[theorem]{Proposition}
\newtheorem{corollary}[theorem]{Corollary}
\theoremstyle{remark}
\numberwithin{equation}{section}
\setlist{itemsep=3pt,topsep=5pt,leftmargin=*}
\hypersetup{pdftitle={Geometry-Guided Proof Distillation for Sendov's Conjecture},pdfauthor={Senjian An},pdfsubject={Geometry of the scalar profile and uniform analytic bounds}}
\title{\color{navy}\textbf{Geometry-Guided Proof Distillation\\[3pt] for Sendov's Conjecture}}
\author{Senjian An\\[3pt]
\small School of Electrical Engineering, Computing and Mathematical Sciences\\
\small Curtin University, Perth, Australia\\
\small\texttt{s.an@curtin.edu.au}}
\date{}

\begin{document}
\maketitle
\begin{abstract}
We investigate the geometry of the scalar optimization problem arising in
the Mazur--Tao reduction of Sendov's conjecture. A differential comparison
produces a smooth cubic polar envelope and a regularized two-parameter
profile $F(m,\alpha)$. Exact and outward-rounded certificates classify its
curvature and stationary points at integer degrees $5\le m\le82$.
The unique interior maximizing locations for $5\le m\le50$ are closely
approximated by a simple increasing, strictly convex rational rule; a
tangent lemma and an analytic tail establish the uniform integer-degree
bound $F<16/25$. Numerical continuation under the rational, cubic, and
exact polar-boundary envelopes reveals similar curvature regimes but
substantial differences in maximum heights and endpoint transitions.
Their common value and first derivative at the origin and their leading
pairwise differences are derived analytically. We then develop several
integral-free majorants. Inspired by Zhang's recent analytic completion,
a split fourth-power kernel-deficit estimate reduces the large-parameter
problem to a concave quadratic and yields, together with Zhang's
small-parameter estimate, $F<95/96$ for every real $m\ge5$ on
$0\le\alpha\le\min\{m/2,17\}$. This analytic completion of the same fixed
scalar profile is independent of the geometric certificates. The focus
is the structure of the optimization landscape and its use in designing
tractable global bounds, rather than priority for a computer-independent
proof of the polynomial theorem.
\end{abstract}

\noindent\textbf{Keywords.} Sendov's conjecture; scalar optimization;
curvature; polar envelope; convex majorant; log-concavity; proof distillation.

\section{Introduction}

For a polynomial $P$ of degree at least two and a zero $a$, write
\[
 d_P(a)=\min\{|a-\xi|:P'(\xi)=0\}.
\]
\begin{theorem}[Sendov]\label{thm:sendov}
If every zero of $P\in\CC[z]$ lies in the closed unit disk, then
$d_P(a)\le1$ for every zero $a$ of $P$.
\end{theorem}

Sendov formulated the conjecture in 1958; his survey~\cite{sendov2002}
and the monograph of Rahman and Schmeisser~\cite{rahmanschmeisser2002}
provide historical background and a broader account of the geometry of
polynomial zeros and critical points. The classical low-degree cases
were established by Brannan for cubics~\cite{brannan1968}, Rubinstein
for cubics and quartics~\cite{rubinstein1968}, and Meir and Sharma for
quintics~\cite{meirsharma1969}. Borcea proved the conjecture for
polynomials with at most six and at most seven distinct
zeros~\cite{borcea1996a,borcea1996b}; Brown gave another proof for
degree seven~\cite{brown1997}. Brown and Xiang subsequently proved it
for every degree at most eight and, more generally, for polynomials
with at most eight distinct zeros~\cite{brownxiang1999}.

In the high-degree direction, D\'egot~\cite{degot2014} proved the
conclusion at a prescribed zero of modulus $0<|a|<1$ once the degree
exceeds a threshold depending on $|a|$. Tao~\cite{tao2022} established
the conjecture for all sufficiently large degrees with an absolute
threshold, independent of the chosen zero.

The all-degree proof architecture used here is due to Mazur~\cite{mazur2026}; we follow the
polynomial reduction in Tao's exposition~\cite{tao2026}. Zhang's
recent preprint~\cite{zhang2026} gives a particularly economical analytic
completion using the refined origin inequality for small values of a
scalar parameter $\alpha$ and the original inequality for larger values.
These works establish the underlying polynomial theorem. The question
here is different: what structure does the refined scalar obstruction
have as a function of its parameters, and how can that structure guide the
construction of simpler bounds? We keep a single substituted scalar
profile throughout the analysis. Appendix~\ref{app:reduction-complete}
records the reduction, including the cap $\alpha\le17$ needed for our
chosen profile; Zhang's proof of Sendov itself does not require this cap.

The preceding arXiv version~\cite{anv1} used a piecewise polar envelope,
eight nodal terms on each of 34 intervals, and log-concavity in the degree.
Subsequent investigations replaced the piecewise envelope by a smooth
cubic function, developed two-chord majorants, and explored the original
integral profile instead of optimizing the majorant terms independently.
The resulting geometric classification identifies changes in curvature,
stationary-point counts, and the global maximizer as the degree varies.
A simple smooth evaluation rule then exploits the concave part of the
profile. In addition to the exact integer-degree conclusions we compare
three envelope choices numerically, separating the robust qualitative
geometry from transitions particular to one envelope.

\paragraph{Main results.}
The first result is a cubic rational polar envelope, proved by differential
comparison. The second is a certified description of the scalar landscape
at $5\le m\le82$, including the unique global maximizing points and a
convex evaluation rule with a justified tangent estimate. Together with an
analytic large-degree tail, it gives the stronger auxiliary bound
$F<16/25$ for integer degrees. A new comparison with the exact polar boundary
and the simpler rational envelope shows that the broad shape is stable
in the numerical study while the endpoint transitions and maximum heights
are envelope-sensitive. The third result is a short analytic completion
of the scalar contradiction on the real-parameter domain. It uses Zhang's
weighted maximum technique with an additional deficit power on the far
half of the integral, and elementary rational estimates for the same
cubic-envelope profile. This gives $F<95/96$ without using the geometric
certificates. Our earlier proof of the separate Jensen--Bernoulli
majorant is retained in the supporting archive.

\paragraph{The role of computation.}
Computer-assisted exploration helped identify the landscape, the useful
comparison functions, and the constants. The fixed-integer curvature
classification is supported by rigorous certificates; the new
three-envelope comparison is numerical and does not certify all roots.
The final analytic completion has no finite parameter certificate or
numerical quadrature as a premise. This distinction concerns logical
verification, not the reliability of the earlier computer-assisted
proofs. We do not claim priority for removing computer assistance from
Sendov's theorem. In particular, the uniform scalar ceilings do not
improve its sharp critical-distance constant one.

\paragraph{Organization.}
Section~\ref{sec:setup} defines the scalar problem, and
Section~\ref{sec:verified-envelope} proves the cubic envelope.
Section~\ref{sec:geometry} develops the geometric results, tangent
estimates, and comparison of envelope choices. Section~\ref{sec:analytic-uniform}
gives the complementary analytic bounds; Section~\ref{sec:completion}
completes the polynomial theorem. We conclude with comparisons and
reproducibility information. The appendices record the reduction,
small-degree proof, analytic tail, and geometric certification methods.

\section{The scalar problem and its exact domain}\label{sec:setup}

Suppose a polynomial with zeros in the closed unit disk has an interior
zero whose critical distances are all at least one. Rotation gives
$0\le a<1$. A zero at the origin is excluded by a product identity,
and the polynomial degrees $2\le n\le5$ are handled in
Appendix~\ref{sec:low-polynomial-degrees}. For $n\ge6$ put $m=n-1$ and
\[
 q_j=\frac1{a-\xi_j},\qquad
 x+iy=\frac1m\sum_{j=1}^m q_j,\qquad
 \alpha=\frac m2(1-a^2),\qquad \eta=1-2ax+a^2.
\]
The complete reduction in Appendix~\ref{app:reduction-complete} proves
\begin{equation}\label{eq:reduction-summary}
 0<\alpha<\frac m2,\quad \alpha\le17,\quad 0<\eta<1,
 \qquad I(\alpha,\eta)>1,\quad 1\le H_m(\alpha,\eta),
\end{equation}
where
\begin{align}
 I(\alpha,\eta)&=\int_0^1e^{\alpha[-1+(2-\eta)t]}\,dt,\\
 B_{m,\alpha,\eta}(t)&=(1-t)^2+\eta t+
                         \frac{2\alpha}{m}t(1-t),\\
 H_m(\alpha,\eta)&=
 \frac{\eta}{2\alpha(1-\eta)}+\frac{\eta}{4\alpha}
 +\frac1{2m}+\frac{\eta}{4\alpha m}\nonumber\\
 &\quad+\frac{\eta(m^2-1)(m-2\alpha)^2}{4\alpha m}
   \int_0^1t^3B_{m,\alpha,\eta}(t)^{(m-3)/2}\,dt.
 \label{eq:original-H}
\end{align}
The two parameters in $H_m$ are independent. The factor $(m-2\alpha)^2$
is squared. These details matter when taking derivatives or comparing
the different proof versions.

Define the cubic envelope and its regularization by
\begin{align}
 u&=\frac{\alpha}{\alpha+3},&
 \ec&=u-\frac{u^3}{30},&
 r&=\frac{29\alpha^2+180\alpha+270}{30(\alpha+3)^3},
 \qquad r(0)=\frac13.
\end{align}
Thus $\ec=\alpha r$. With $p=(m-3)/2$ and
$A_m=\min\{m/2,17\}$, the substituted profile on $[0,A_m]$ is
\begin{align}
 B(t)&=(1-t)^2+\ec t+\frac{2\alpha}{m}t(1-t),&
 J&=\int_0^1t^3B(t)^p\,dt,\nonumber\\
 E&=\frac{r}{2(1-\ec)}+\frac r4+\frac1{2m}+\frac r{4m},&
 K&=\frac{m^2-1}{4m}r(m-2\alpha)^2,\nonumber\\
 F(m,\alpha)&=E+KJ=H_m(\alpha,\ec(\alpha)).
 \label{eq:profile}
\end{align}
At zero the last expression means the continuous extension of the
substituted profile. Its value is
\[
 F(m,0)=\frac14+\frac7{12m}+\frac1{2(m-2)}.
\]
At $\alpha=m/2$ the prefactor $K$ vanishes and $F=E$.
All subsequent $\alpha$ derivatives of $F$ include the dependence of
$\ec$ on $\alpha$.

Section~\ref{sec:verified-envelope} proves $\eta<\ec$ from the polar constraint and strict
increase of $H_m$ in $\eta$. Hence a hypothetical counterexample forces
$1\le H_m<F$. The geometry in Section~\ref{sec:geometry} gives a stronger integer-degree bound, while Section~\ref{sec:analytic-uniform} supplies the uniform analytic bound $F<95/96$.

\section{An analytic cubic envelope and the monotonicity transfer}
\label{sec:verified-envelope}

This section proves the cubic polar envelope directly.  The only hypothesis
imported from the polynomial reduction is the polar inequality itself; the
hyperbolic estimate and its permitted range are proved below.

\subsection{A short hyperbolic comparison}

\begin{lemma}\label{lem:short-hyperbolic}
For $0<x\le10$,
\[
 \coth x-\frac1x<
 \frac{x}{\sqrt{9+8x^2/7}}.
\]
Consequently, for $0\le x\le10$,
\[
 \log\frac{\sinh x}{x}\le
 \frac78\left(\sqrt{9+\frac87x^2}-3\right),
\]
where the quotient at zero is interpreted continuously, and the inequality
is strict for $x>0$.
\end{lemma}

\begin{proof}
Put
\[
 h(x)=\coth x-\frac1x,\qquad
 v(x)=\frac{x}{\sqrt{9+8x^2/7}},\qquad
 w(x)=\sqrt{9+8x^2/7},\qquad D=h-v.
\]
The identities
\[
 h'=1-h^2-\frac{2h}{x},\qquad v'=\frac9{w^3}
\]
give
\begin{equation}\label{eq:riccati-comparison}
 D'+\left(h+v+\frac2x\right)D
 =\frac{(w-3)(w^2-13w+24)}{8w^3}=:R(x).
\end{equation}
For $x>0$ we have $w>3$.  The smaller root of
$w^2-13w+24$ is $(13-\sqrt{73})/2<3$, and the larger root
is $(13+\sqrt{73})/2>3$.  Thus $R$ changes sign exactly once
from negative to positive as $x$ increases over $(0,\infty)$.

Put
\[
 G(x)=\frac78\left(\sqrt{9+\frac87x^2}-3\right),\qquad
 \mu(x)=x\sinh x\,e^{G(x)}.
\]
Since $G'=v$, we have $\mu'/\mu=h+v+2/x$ for $x>0$.
Thus $W:=\mu D$ satisfies $W'=\mu R$ and decreases and then
increases.  Moreover,
\[
 W(x)=e^{G(x)}
 \left(x\cosh x-\sinh x-\frac{x^2\sinh x}{w(x)}\right)
\]
extends continuously to $W(0)=0$.

At the other endpoint,
\[
 e^{20}>\sum_{k=0}^{4}\frac{20^k}{k!}=8221>4001,
\]
and hence
\[
 h(10)=\frac9{10}+\frac2{e^{20}-1}
 <\frac{1801}{2000}.
\]
On the other hand,
\[
 v(10)^2=\frac{700}{863}>
 \left(\frac{1801}{2000}\right)^2,
\]
because
$700\cdot2000^2-863\cdot1801^2=772337>0$.
Therefore $D(10)<0$.

It follows that $W(10)<0$.  Because $W'$ is initially strictly
negative and changes sign at most once, from negative to positive,
the endpoint values $W(0)=0$ and $W(10)<0$ imply $W(x)<0$ for
every $0<x\le10$.  Since $\mu(x)>0$ there, $D(x)<0$ follows.
Finally,
\[
 \frac{d}{dx}\log\frac{\sinh x}{x}=h(x),\qquad
 \int_0^x v(s)\,ds
 =\frac78\left(\sqrt{9+\frac87x^2}-3\right).
\]
Both primitives vanish at zero, which proves the stated integral bound
and its strictness for positive $x$.
\end{proof}

\subsection{The cubic polar envelope}

\begin{theorem}\label{thm:verified-cubic-envelope}
Let $0<\alpha\le17$, and suppose
\[
 I(\alpha,\eta)
 :=\int_0^1
    \exp\bigl(\alpha[-1+(2-\eta)t]\bigr)\,dt>1.
\]
Then
\[
 \eta<\eta_{\mathrm c}(\alpha)
 :=u-\frac{u^3}{30},\qquad
 u=\frac{\alpha}{\alpha+3}.
\]
In particular, this applies to every positive polar parameter permitted
by the stated reduction.
\end{theorem}

\begin{proof}
For fixed $\alpha>0$,
\[
 I_\eta=-\alpha\int_0^1t
           e^{\alpha[-1+(2-\eta)t]}\,dt<0,
\]
and $I(\alpha,1)=(1-e^{-\alpha})/\alpha<1$.
Thus $I>1$ implies $\eta<1$.
Suppose, towards a contradiction, that
$\eta\ge\eta_{\mathrm c}(\alpha)>0$.  Define
\[
 p=\frac{\alpha\eta}{2},\qquad
 x=\alpha-p=\alpha\left(1-\frac{\eta}{2}\right)>0.
\]
The polar integral has the exact form
\[
 I(\alpha,\eta)=e^{-p}\frac{\sinh x}{x};
\]
therefore $p<\log(\sinh x/x)$.

We first verify the range required by Lemma~\ref{lem:short-hyperbolic}.
Put $\psi(\alpha)=\alpha(1-\eta_{\mathrm c}(\alpha)/2)$.
Since $u'=(1-u)^2/3$ and
$\eta_{\mathrm c}'=(1-u^2/10)u'$, direct substitution gives
\[
 2\psi'(\alpha)-1
 =(1-u)^2+\frac{u^3(4-3u)}{30}>0.
\]
Hence $\psi$ is increasing, and under the contradiction hypothesis
\[
 0<x\le\psi(\alpha)\le\psi(17)
 =\frac{4775521}{480000}<10.
\]
The hyperbolic estimate now yields
\[
 p<\frac78
       \left(\sqrt{9+\frac87(\alpha-p)^2}-3\right).
\]
Both sides of the equivalent inequality
$3+8p/7<\sqrt{9+8(\alpha-p)^2/7}$ are positive.
Squaring and simplifying gives
\[
 Q_\alpha(p):=\frac{p^2}{7}+2(\alpha+3)p-\alpha^2<0.
\]
But $Q_\alpha$ is strictly increasing for $p\ge0$, and
$p\ge p_0:=\alpha\eta_{\mathrm c}(\alpha)/2$.  At this point,
\[
 \frac{Q_\alpha(p_0)}{\alpha^2}
 =\frac{u^2(u^4-60u^2+60)}{25200}>0.
\]
Indeed,
\[
 u^4-60u^2+60
 =1+(1-u^2)(59-u^2)>1
 \qquad(0<u<1).
\]
This contradicts $Q_\alpha(p)<0$ and proves the envelope.
\end{proof}

\subsection{Strict monotonicity of the original obstruction}

For $m\ge5$, $\alpha>0$ and $0<\eta<1$, define
\begin{align*}
 B(t)&=(1-t)^2+\eta t+\frac{2\alpha}{m}t(1-t),&
 p_m&=\frac{m-3}{2},\\
 J_m&=\int_0^1t^3B(t)^{p_m}\,dt,&
 c_{m,\alpha}&=\frac{(m^2-1)(m-2\alpha)^2}{4\alpha m}.
\end{align*}
The original functional is
\[
 H_m(\alpha,\eta)
 =\frac{\eta}{2\alpha(1-\eta)}
   +\frac{m+1}{4\alpha m}\eta+\frac1{2m}
   +c_{m,\alpha}\eta J_m.
\]
The kernel is positive on $[0,1]$, and differentiation in $\eta$ gives
\[
 (J_m)_\eta
 =p_m\int_0^1t^4B(t)^{p_m-1}\,dt>0.
\]
Consequently,
\[
 (H_m)_\eta
 =\frac1{2\alpha(1-\eta)^2}
   +\frac{m+1}{4\alpha m}
   +c_{m,\alpha}\bigl(J_m+\eta(J_m)_\eta\bigr)>0.
\]
This is strict even at $\alpha=m/2$, where $c_{m,\alpha}=0$.
Combining it with Theorem~\ref{thm:verified-cubic-envelope} gives
\[
 H_m(\alpha,\eta)
 <H_m(\alpha,\eta_{\mathrm c}(\alpha))
 =:F(m,\alpha)
\]
for every actual positive polar parameter with $0<\alpha\le17$.
This transfer uses the original functional with independent $\eta$;
derivatives of $F$ in $\alpha$ must separately include the dependence
of $\eta_{\mathrm c}$ on $\alpha$.

\subsection{The added endpoints of the substituted profile}

Set
\[
 r(\alpha)=\frac{\eta_{\mathrm c}(\alpha)}{\alpha}
 =\frac{29\alpha^2+180\alpha+270}{30(\alpha+3)^3},
 \qquad r(0)=\frac13.
\]
This removes the apparent singularity at $\alpha=0$ in the substituted
profile.  For each fixed real $m\ge5$, its elementary and prefactor
limits are
\[
 E(m,0)=\frac14+\frac7{12m},\qquad
 K(m,0)=\frac{m(m^2-1)}{12}.
\]
The kernels converge continuously on the compact $t$ interval to
$(1-t)^2$.  Since $p_m\ge1$, continuity of the power and dominated
convergence give
\[
 J(m,0)=\int_0^1t^3(1-t)^{m-3}\,dt
       =\frac6{(m-2)(m-1)m(m+1)}.
\]
Thus the continuous extension satisfies
\[
 F(m,0)=\frac14+\frac7{12m}+\frac1{2(m-2)}.
\]
This extends $F(m,\alpha)=H_m(\alpha,\eta_{\mathrm c}(\alpha))$;
it does not define $H_m(0,\eta)$ for an independent nonzero $\eta$.
No differentiability of arbitrary order at $\alpha=0$ is required.

If the other endpoint is $\alpha=m/2$, then $K=0$ and $F=E$.
The integral remains finite, so the product is evaluated without
division by $K$.  At a cap endpoint $\alpha=17$ the original formulas
are continuous; $K$ vanishes there only when $m=34$.
Hence $F$ is continuous on
$[0,\min(m/2,17)]$, and interior strict monotonicity in $\eta$ is fully
compatible with using that compact interval for the scalar bounds.

\subsection{Verification of the analytic identities}

The proof above is analytic.  The accompanying exact-arithmetic verifier
checks its identities, coefficient and endpoint signs, and the formulas
for the continuous extension.  It also checks the longer hyperbolic
series certificate preserved with the earlier work as an alternative.  No numerical quadrature, finite sampling in $\alpha$,
or assumed fit to the polar boundary enters this envelope proof.
The polynomial-to-scalar reduction is proved separately in Appendix~\ref{app:reduction-complete}.

\section{Geometry of the cubic-envelope scalar profile}
\label{sec:geometry}

Our principal geometric results describe the curvature and stationary
points of the original integral profile, justify a smooth tangent rule,
and supply the auxiliary ceiling $16/25$. The integer-degree statements
proved by certificates are distinguished from the numerical coordinates
used in the figures. The analytic completion in Section~\ref{sec:analytic-uniform}
is independent of these certificates; its purpose is to distill the
scalar contradiction rather than to locate the extrema precisely.

Throughout, the envelope is
\[
 \ec(\alpha)=u-\frac{u^3}{30},\qquad
 u=\frac{\alpha}{\alpha+3},\qquad
 r(\alpha)=\frac{29\alpha^2+180\alpha+270}{30(\alpha+3)^3},
 \quad r(0)=\frac13.
\]
The profile is
\begin{align*}
 F(m,\alpha)&=E+KJ,\qquad 0\le\alpha\le A_m:=\min\{m/2,17\},\\
 E&=\frac{r}{2(1-\ec)}+\frac r4+\frac1{2m}+\frac r{4m},\\
 K&=\frac{m^2-1}{4m}r(m-2\alpha)^2,\qquad
 J=\int_0^1t^3B(t)^{(m-3)/2}\,\dd t,\\
 B(t)&=(1-t)^2+\ec t+\frac{2\alpha}{m}t(1-t).
\end{align*}
Every derivative in \(\alpha\) below includes the dependence of the envelope.
The geometric results concern this cubic-envelope profile; they do not
automatically transfer to the piecewise-envelope profile of the preceding
arXiv version.

\subsection{A complete classification at the certified integer degrees}

\begin{theorem}[Certified geometry]\label{thm:geometry-classification}
For every integer \(5\le m\le82\), the curvature, stationary points, and
global maximizing point of \(F(m,\cdot)\) have the following classification.
All inflections listed in the table are simple.
\[
\begin{array}{c|c|c|c|l}
 m&\text{signs of }F_{\alpha\alpha}
   &\text{interior minima}&\text{interior maxima}
   &\text{unique global maximizing point}\\ \hline
 5&-,+&0&1&\alpha_*(m)\text{ in the interior}\\
 6,\ldots,13&+,-,+&0&1&\alpha_*(m)\text{ in the interior}\\
 14,\ldots,39&+,-,+&1&1&\alpha_*(m)\text{ in the interior}\\
 40,\ldots,50&+,-&1&1&\alpha_*(m)\text{ in the interior}\\
 51,\ldots,66&+,-&1&0&17\\
 67,\ldots,81&+&1&0&17\\
 82&+&1&0&0
\end{array}
\]
For the prescribed rule
\[
 \ell(m)=\frac{441m}{1000}-\frac{177}{20}+\frac{354}{m+45},
\]
the interior maximizing points satisfy
\[
 0<\alpha_*(m)<\ell(m)<A_m
 \qquad(m=5,\ldots,50).
\]
This statement identifies a unique maximizing point at each stated integer
degree. It does not give a closed formula for that point.
\end{theorem}

\begin{proof}
The complete curvature partitions are established by the accompanying
exact polynomial and outward-rounded interval certificates. Odd degrees
use rational profiles and complete Bernstein root covers. Even degrees
use strict second-derivative signs on a complete complementary cover,
opposite endpoint signs on each inflection bracket, and a strict
third-derivative sign throughout that bracket. At \(m=6\), the positive
origin curvature is proved by a one-sided lower bound and continuity,
without assuming that \(F_{\alpha\alpha}(m,0)\) is finite.

Only small additional checks are needed to deduce the stationary counts
from those curvature partitions. The endpoint identity
\begin{equation}\label{eq:geometry-origin-slope}
 F_\alpha(m,0)=
 -\frac{m^3-11m^2-34m+8}{36m(m-2)(m-4)}
 \qquad(m>4)
\end{equation}
follows by differentiating the regularized profile and evaluating beta
integrals. It is positive at integers \(5,\ldots,13\) and negative for
every real \(m\ge14\). For the latter assertion, the numerator is positive
at \(14\) and has positive, increasing derivative thereafter.
The remaining certified signs and comparisons are
\begin{align}
 F_\alpha(m,A_m)&<0,
 \quad F_\alpha(m,\ell(m))<0,
 \quad F(m,\ell(m))>F(m,0)
 &&(5\le m\le50),\label{eq:geometry-new-gates}\\
 F_\alpha(m,17)&>0
 &&(51\le m\le82).\nonumber
\end{align}
The first line uses 46 recorded rigorous right-endpoint derivative enclosures
and the already certified anchor values. Its comparisons are exact
integer or rational comparisons, with integration remainders retained.

For \(m=5,\ldots,13\), the derivative begins positive, crosses zero once
on the concave interval, and remains negative afterward. On any final
convex interval it is increasing toward a still-negative endpoint value.
Thus the profile increases strictly to one interior maximum and then
decreases strictly.

For \(m=14,\ldots,50\), the derivative begins negative. The strict height
comparison in \eqref{eq:geometry-new-gates} and the mean value theorem
force a positive derivative somewhere before \(\ell\). The curvature
pattern then forces exactly one crossing from negative to positive,
followed by exactly one crossing from positive to negative. These are a
minimum and a maximum, respectively. The profile decreases strictly
after the maximum; its height exceeds \(F(m,\ell)>F(m,0)\), so it is the
unique global maximum. Since the derivative is negative at \(\ell\),
that maximum lies strictly to its left.

For \(m=51,\ldots,66\), the derivative begins negative, increases to the
inflection, and then decreases to its positive endpoint value. It therefore
has exactly one zero, a minimum. For \(m=67,\ldots,82\), strict convexity
and the opposite endpoint derivative signs give the same count.
The global maximum must be at an endpoint. The previously certified
strict endpoint comparisons select \(17\) through \(m=81\) and \(0\)
at \(m=82\), completing the classification.
\end{proof}

The table explains why a claim of one inflection is not an accurate
description of the full finite range. For \(m=6,\ldots,39\) there are two:
one convex-to-concave transition and one concave-to-convex transition.
There is no interior minimum at \(m=5,\ldots,13\). At the upper end,
the constrained maximizing point is fixed at \(17\) for 31 integer
degrees and then changes to the other endpoint at \(m=82\).
No optimizer statement for \(m>82\) is inferred from this finite table.

\begin{figure}[htbp]
\centering
\includegraphics[width=\textwidth]{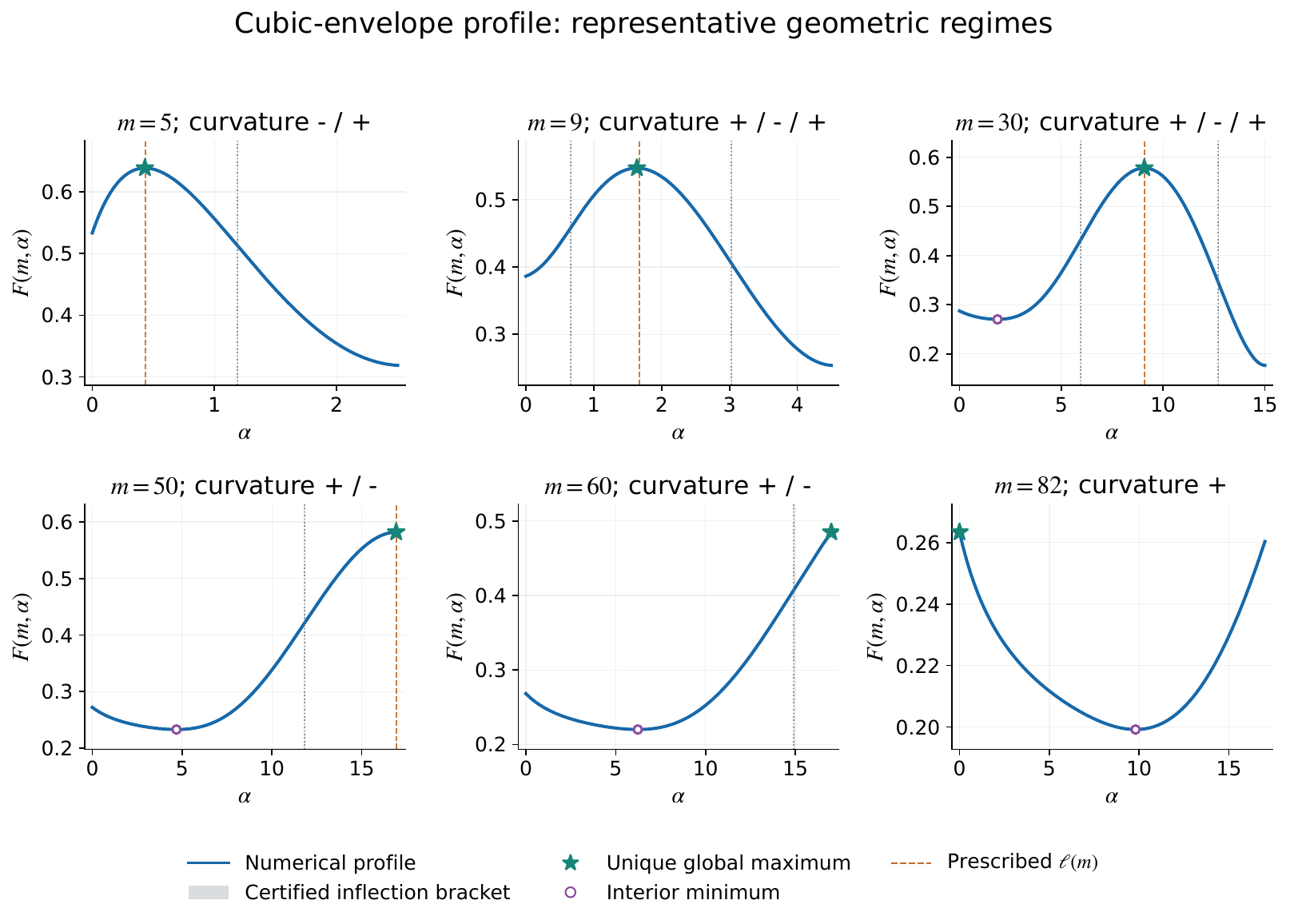}
\caption{Representative profiles computed from the exact defining formula
by numerical quadrature. The gray bands are the certified rational
inflection brackets; dotted center lines indicate numerical root locations.
Stars indicate the unique global maxima, and open circles the interior
minima. Their existence, uniqueness, and endpoint classifications are
certified; the plotted interior coordinates and curve heights are numerical.
The orange line marks the prescribed \(\ell(m)\) where it is used.}
\label{fig:geometry-profiles}
\end{figure}

\begin{figure}[htbp]
\centering
\includegraphics[width=\textwidth]{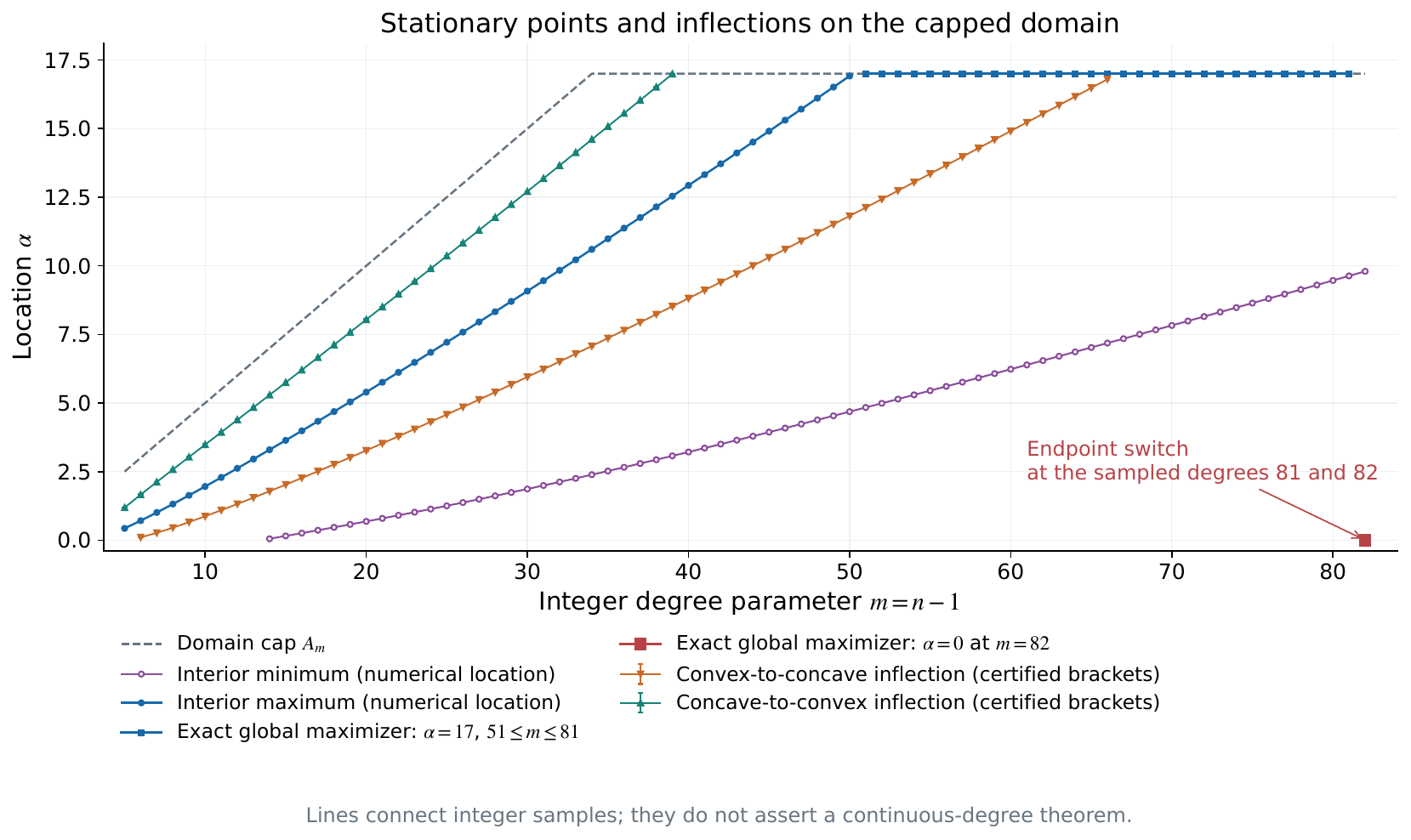}
\caption{Characteristic locations at integer degrees \(m=5,\ldots,82\).
Interior stationary coordinates are numerical approximations to the
uniquely classified stationary points. Inflection error bars show exact
rational brackets, with numerical locations inside them. The endpoint
maximizers at \(51\le m\le81\) and \(m=82\) are exact. Lines connecting
integer samples are visual guides and do not assert a theorem for all
intermediate real degrees.}
\label{fig:characteristic-locations}
\end{figure}

\begin{samepage}
\subsection{The prescribed convex rule and the true maximizing locations}

\begin{proposition}[Geometry of the prescribed anchor]
\label{prop:smooth-anchor-geometry}
The selected curve \(\ell\) is strictly increasing, strictly convex, and
admissible on the whole real interval \(5\le m\le50\). At every integer
in that interval, it lies strictly to the right of the unique maximizing
point and inside the certified concave portion of the profile.
\end{proposition}
\end{samepage}
\begin{proof}
The selected curve has exact derivatives
\[
 \ell'(m)=\frac{441}{1000}-\frac{354}{(m+45)^2}>0,\qquad
 \ell''(m)=\frac{708}{(m+45)^3}>0
 \qquad(5\le m\le50).
\]
Its admissibility on that whole real interval also follows directly:
\(\ell(5)=87/200>0\), \(\ell(50)=1608/95<17\), and
\(m/2-\ell(m)\) is positive and increasing. These are continuous-real-\(m\)
facts about the prescribed formula.
The integer ordering follows from
Theorem~\ref{thm:geometry-classification}; the location of \(\ell\)
between the rational concavity endpoints is checked in every finite
certificate.
\end{proof}

The corresponding facts about the true maximizer have a narrower scope.
Theorem~\ref{thm:geometry-classification} certifies
\(\alpha_*(m)<\ell(m)\) at the 46 integer degrees. Numerical calculations
place the displacement between approximately \(0.00422\) and \(0.03910\).
All 44 available centered second differences of the sampled maximizing
locations are positive. The displacement sizes and second differences
are numerical observations, and they do not prove convexity of a
maximizing branch on the continuous interval \(5\le m\le50\).

The distinction can be stated precisely by implicit differentiation.
On any smooth nondegenerate stationary branch satisfying
\(F_\alpha(m,\alpha_*(m))=0\),
\[
 \alpha_*'=-\frac{F_{\alpha m}}{F_{\alpha\alpha}},\qquad
 \alpha_*''=
 -\frac{F_{\alpha mm}+2F_{\alpha\alpha m}\alpha_*'
             +F_{\alpha\alpha\alpha}(\alpha_*')^2}
 {F_{\alpha\alpha}}.
\]
At a nondegenerate maximum, \(F_{\alpha\alpha}<0\); this alone does not
determine the sign of the mixed-derivative numerator. The curvature of
\(F\) as a function of \(\alpha\), the curvature of a stationary location
as a function of \(m\), and the curvature of \(\ell\) are three separate
questions. The fitted constants in \(\ell\) are useful choices, not
constants derived from the stationary equation.

\begin{figure}[htbp]
\centering
\includegraphics[width=.94\textwidth]{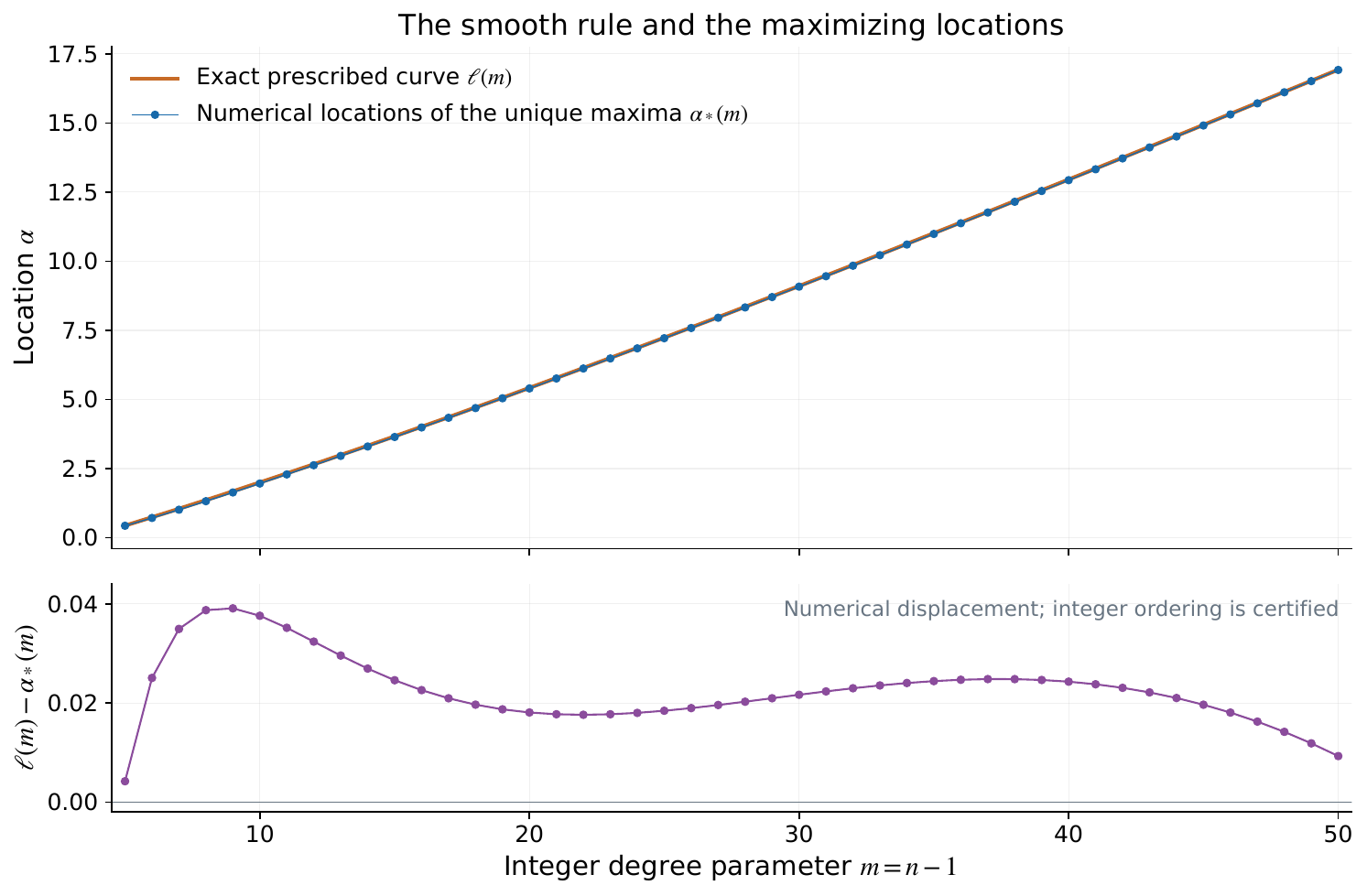}
\caption{The exact prescribed curve \(\ell(m)\), numerical locations of the
unique interior maxima, and their numerical displacement. Convexity of
\(\ell\) is proved on the real interval. The ordering
\(\alpha_*(m)<\ell(m)\) is certified at every plotted integer degree.
The numerical samples do not establish continuous-\(m\) convexity of
\(\alpha_*\).}
\label{fig:smooth-anchor-fit}
\end{figure}

\subsection{Why a close anchor and a simple majorant are both useful}

\begin{lemma}[Tangent bound with convex outer pieces]
\label{lem:geometry-tangent}
Let \(f\) be continuously differentiable on \([0,A]\), convex on
\([0,b]\) and \([c,A]\), and concave on \([b,c]\), where
\(0\le b\le\ell\le c\le A\). Degenerate outer intervals are allowed.
If \(f'(\ell)\le0\), then
\[
 \max_{0\le x\le A}f(x)
 \le\max\{f(0),f(A),f(\ell)-\ell f'(\ell)\}.
\]
\end{lemma}
\begin{proof}
On \([b,c]\), concavity bounds \(f(x)\) by the tangent
\(f(\ell)+f'(\ell)(x-\ell)\). Since \(x\ge0\) and
\(f'(\ell)\le0\), this is at most \(f(\ell)-\ell f'(\ell)\).
Convexity bounds each outer piece by its two endpoint values. The inner
endpoint values have already been bounded by the tangent intercept,
so the three displayed quantities bound the whole interval.
\end{proof}

The complete curvature classification and the anchor placement apply
this lemma to the present profiles. Thus
\[
 \max_{0\le\alpha\le A_m}F(m,\alpha)
 \le\max\{F(m,0),F(m,A_m),\Theta(m)\},\qquad
 \Theta(m)=F(m,\ell)-\ell F_\alpha(m,\ell).
\]
The certificate proves that \(\Theta(m)\) exceeds both endpoint values.
Thus a convenient evaluation point suffices; solving the stationary
equation exactly is unnecessary for the tangent bound.

There is a second useful simplification. Put \(h=3/5\),
\[
 b=B(h)=\frac4{25}+\frac35\ec+\frac{12\alpha}{25m},\qquad
 p=\frac{m-3}{2},\qquad
 A_h=\frac{259}{3125},\quad B_h=\frac{421}{3125}.
\]
The two-chord pointwise majorant has the form
\begin{align*}
 G(m,\alpha)&=E+T_0+T_1+T_2,\\
 T_0&=\frac{24h^4r(m-2\alpha)^2}
 {m(m+3)(m+5)(1-b)^4},\\
 T_1&=A_hK b^p,\qquad T_2=B_hK\ec^p.
\end{align*}
For real \(m\ge5\) on the capped domain, \(B\) is convex in \(t\),
since $B''(t)=2(1-2\alpha/m)\ge0$. On $[0,h]$, the first chord gives
$B(t)\le1-(1-b)t/h$, so
\[
 \int_0^h t^3B(t)^p\,\dd t
 \le h^4\int_0^{1/(1-b)}v^3(1-(1-b)v)^p\,\dd v
 =\frac{6h^4}{(1-b)^4\prod_{j=1}^4(p+j)}.
\]
Multiplication by $K$ gives $T_0$. On $[h,1]$, the second chord and
convexity of the power $s\mapsto s^p$ give
\[
 B(t)^p\le\frac{1-t}{1-h}b^p+\frac{t-h}{1-h}\ec^p.
\]
The two weights have integrals
\[
 \int_h^1\frac{t^3(1-t)}{1-h}\,\dd t=A_h,\qquad
 \int_h^1\frac{t^3(t-h)}{1-h}\,\dd t=B_h.
\]
Consequently $F\le G$. This historical two-chord majorant is used here to
interpret the retained comparison figure; the analytic completion uses
the weighted kernel-deficit majorant of Section~\ref{sec:analytic-uniform}.

This creates two distinct sources of slack:
\[
 \max F\le\max(E+T_0+T_1+T_2)
 \le\max E+\max T_0+\max T_1+\max T_2.
\]
The pointwise two-chord approximation is one relaxation. Bounding the
four terms separately is another, because their maximizing locations
need not agree. The first is illustrated in
Figure~\ref{fig:profile-majorant}. The previous forty-interval argument
reduced the second by mixed endpoint estimates. The preceding analytic proof controlled a Jensen--Bernoulli majorant.
The shortened argument below instead combines an affine kernel near zero
with a fourth-power weighted estimate on the far half.

\begin{figure}[htbp]
\centering
\includegraphics[width=\textwidth]{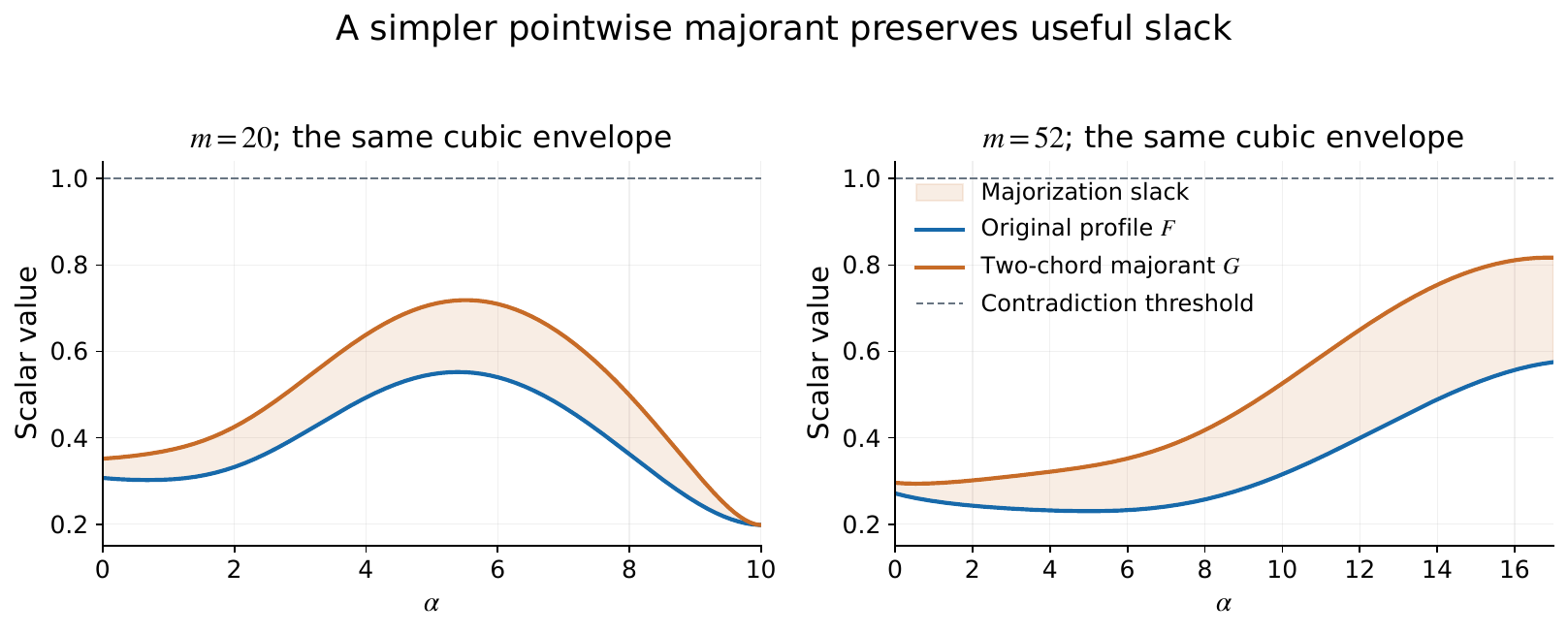}
\caption{Numerical profiles \(F\) and the explicit pointwise majorant \(G\)
at two representative degrees, using the same cubic envelope. The shaded
gap illustrates the first relaxation in the displayed maximum inequalities.
These sampled curves do not certify a global maximum of \(G\) over all
degrees or all real parameters. The validity of \(F\le G\) follows from
the analytic chord argument, independently of the plot.}
\label{fig:profile-majorant}
\end{figure}

\subsection{A stronger auxiliary scalar ceiling}

\begin{corollary}\label{cor:global-geometry64}
Retaining the finite geometric certificates and the short analytic tail
of Theorem~\ref{thm:tail64}, the cubic-envelope profile satisfies
\[
 \boxed{F(m,\alpha)<\frac{16}{25}
 \quad\text{for every integer }m\ge5,\quad
 0\le\alpha\le\min\{m/2,17\}.}
\]
\end{corollary}
\begin{proof}
For \(m=5,\ldots,50\), the smooth tangent certificates give
\(F<639863/10^6<16/25\). For \(m=51,\ldots,80\), the endpoint
certificates give \(F<29/50<16/25\). For every real \(m\ge81\),
Theorem~\ref{thm:tail64} gives
\(F<77663/121500<16/25\). The ranges cover every required integer.
\end{proof}

This stronger scalar bound is an additional consequence of the geometric
analysis. The separate analytic proof uses the relaxed ceiling $95/96$ and
requires none of these geometric certificates or their separate tail.
A smaller auxiliary scalar ceiling does not by itself replace the
distance one in the polynomial theorem.

\subsection{Reproducibility and status of the figures}

The file \texttt{geometry/certificates/geometry\_certificate.json} records
all 78 fixed-degree classifications, exact zero-endpoint slopes, the
46 added right-endpoint slope enclosures, and the exact comparisons
\(F(m,\ell)>F(m,0)\). The command
\texttt{python geometry/verify\_geometry.py --proof-root proof}
reproduces these additional gates from the retained geometric proof package.
The original complete curvature certificates remain the inputs for the
curvature patterns.

The separate script \texttt{geometry/make\_geometry\_figures.py} evaluates
the exact stated formula with 384-node Gauss--Legendre quadrature and
refines numerical roots inside certified geometric regions. It writes
the vector PDF figures, PNG previews, numerical profile samples, stationary
locations, and the rational inflection brackets. A second quadrature order
is used only as a numerical diagnostic. All numerical locations,
displacements, second differences, and plotted heights retain that
status; neither quadrature agreement nor interpolation of integer
samples is treated as a proof.

\subsection{Comparison with the earlier polar envelope}
\label{subsec:v1-envelope-comparison}

The preceding arXiv version uses the piecewise polar envelope
\[
 \eta_*(\alpha)=
 \begin{cases}
  \alpha/(\alpha+3),&0<\alpha\le11,\\
  1-\log\alpha/\alpha,&11<\alpha\le17,
 \end{cases}
\]
as stated in its equation (8). The subsequent two-chord, piecewise-anchor,
smooth-anchor, and Bernoulli arguments use the cubic envelope
\(\ec=u-u^3/30\). Thus comparisons among these routes distinguish the
chosen envelope from the method used to bound its substituted profile.
This comparison provides context for the earlier versions; the present large-parameter analytic
completion uses the cubic envelope directly.

\begin{proposition}[Pointwise comparison of the envelopes]
\label{prop:envelope-comparison}
For every \(0<\alpha\le17\),
\[
 \ec(\alpha)<\eta_*(\alpha).
\]
Consequently, for the same admissible \(m\ge5\) and \(\alpha>0\),
\[
 H_m(\alpha,\ec(\alpha))<H_m(\alpha,\eta_*(\alpha)).
\]
\end{proposition}
\begin{proof}
On the rational branch,
\(\eta_*-\ec=u^3/30>0\).
For the logarithmic branch put
\(\phi(\alpha)=\alpha(1-\ec(\alpha))-\log\alpha\).
Direct differentiation gives
\[
 \phi'(\alpha)=
 \frac{\alpha^4(\alpha-18)-90\alpha^3-810\alpha-2430}
 {30\alpha(\alpha+3)^4}<0
 \qquad(0<\alpha\le17).
\]
The exponential series supplies the rational comparison
\[
 e^{57/20}>
 \sum_{k=0}^{7}\frac{(57/20)^k}{k!}
 =17+\frac{95055949517}{716800000000}>17,
\]
so \(\log17<57/20\). Hence
\[
 \phi(\alpha)\ge\phi(17)>
 \frac{695521}{240000}-\frac{57}{20}
 =\frac{11521}{240000}>0.
\]
This is precisely \(\ec<1-\log\alpha/\alpha\).
The profile comparison follows from strict increase of \(H_m\) in \(\eta\).
\end{proof}

This proposition shows why an improved final ceiling cannot be attributed
solely to a different optimization method: the cubic-envelope profile is
already pointwise smaller. Comparisons among the later two-chord, tangent,
and Bernoulli methods keep the cubic envelope fixed. The exact algebra and
rational comparisons above are optionally reproduced by
\texttt{geometry/verify\_envelope\_comparison.py}.

\subsection{Sensitivity of the landscape to the polar envelope}
\label{subsec:three-envelope-study}

A basic issue in interpreting the preceding geometry is whether it is
specific to the cubic upper bound. To separate the influence of the
polar reduction from that of its approximation, define the exact polar
boundary $\eta_{\rm pol}(\alpha)$ by
\[
 I(\alpha,\eta_{\rm pol}(\alpha))=1,
 \qquad \alpha>0.
\]
For fixed $\alpha>0$, the function $I$ is strictly decreasing in $\eta$;
$I(\alpha,0)=\sinh\alpha/\alpha>1$ and
$I(\alpha,1)=(1-e^{-\alpha})/\alpha<1$. Thus the boundary exists uniquely.
Compare the three profiles on the same capped domain:
\[
 F_0=H_m(\alpha,u),\qquad F=H_m(\alpha,\ec),\qquad
 F_{\rm pol}=H_m(\alpha,\eta_{\rm pol}).
\]
This notation distinguishes the exact polar boundary from the piecewise
upper envelope of the preceding arXiv version. The proof of
Theorem~\ref{thm:verified-cubic-envelope} also works when $I=1$:
the strict hyperbolic inequality still yields the strict contradiction.
Hence for $0<\alpha\le17$ and $m\ge5$,
\begin{equation}\label{eq:three-profile-order}
 \eta_{\rm pol}<\ec<u,\qquad F_{\rm pol}<F<F_0.
\end{equation}
This is an analytic ordering of the profiles and entails the non-strict
ordering of their maximum heights. It does not impose an ordering of the
maximizing locations.

\paragraph{Local agreement at the origin.}
For the exact polar boundary let $L(h)=\log(\sinh h/h)$ and
$\tau=\alpha\eta_{\rm pol}/2$. The equality $I=1$ is equivalent to
\[
 \alpha=h+L(h),\qquad \tau=L(h),\qquad h=\alpha-\tau.
\]
Since $L'(h)=\coth h-1/h>0$ for $h>0$, this gives a monotone
one-dimensional parameterization. The analytic implicit function theorem
and $L(h)=h^2/6-h^4/180+O(h^6)$ give
\[
 \tau=\frac{\alpha^2}6-\frac{\alpha^3}{18}
          +\frac{19\alpha^4}{1080}+O(\alpha^5).
\]
In particular the three regularized ratios $r_i=\eta_i/\alpha$ have the
expansions
\begin{align*}
 r_0&=\frac13-\frac\alpha9+\frac{\alpha^2}{27}+O(\alpha^3),\\
 r_{\rm c}&=\frac13-\frac\alpha9+\frac{29\alpha^2}{810}+O(\alpha^3),\\
 r_{\rm pol}&=\frac13-\frac\alpha9+\frac{19\alpha^2}{540}+O(\alpha^3).
\end{align*}
Thus all three profiles share the exact endpoint value and first
derivative in Section~\ref{sec:geometry}. More precisely, write
\[
 A(m)=\frac34+\frac1{4m}+\frac{3}{2(m-2)}.
\]
Treating the regularized ratio as an independent variable and substituting
$\eta=\alpha r$, the derivative of the resulting profile with respect to
$r$ converges to $A(m)$ as $\alpha\downarrow0$. This follows by dominated
convergence in the integral and its $r$ derivative: $\partial_r B=\alpha t$,
$p\ge1$, and $\partial_r K=m(m^2-1)/4$ at zero. The beta integral for
$J(m,0)$ gives the last term of $A(m)$. The mean value theorem therefore
proves, for each fixed real $m\ge5$,
\begin{align*}
 F_0-F&=\frac{A(m)}{810}\alpha^2+o(\alpha^2),\\
 F-F_{\rm pol}&=\frac{A(m)}{1620}\alpha^2+o(\alpha^2).
\end{align*}
This also explains why a sharper envelope has little influence very
close to zero. The difference estimates do not require that each
individual second derivative have a finite limit there.

\paragraph{Numerical continuation across degrees.}
We compare all integers $5\le m\le120$ using the common domain
$[0,A_m]$. The exact-boundary parameterization above is convenient for
computing derivatives: $\tau'=L'(h)/(1+L'(h))$ and
$\tau''=L''(h)/(1+L'(h))^3$. The first and second derivatives of the
profiles are evaluated by differentiating their defining expressions;
near zero, the regularized polar ratio is evaluated by a Taylor
expansion to avoid cancellation. Each scan uses 1001 uniformly spaced
points and additional logarithmically spaced points near zero, followed
by refinement of detected sign changes and comparison with both
endpoints. Runs with 192 and 384 Gauss--Legendre nodes agree to within
$10^{-12}$ in the reported maximum heights and locations, and give the
same detected curvature and maximizing-point regimes. These are
numerical diagnostics, not exhaustive root counts. The previously
certified cubic conclusions through $m=82$ remain the rigorous
classification of Theorem~\ref{thm:geometry-classification}.

\begin{center}\small
\begin{tabular}{@{}lccc@{}}
\toprule
Observed curvature signs & Rational & Cubic & Polar boundary\\
\midrule
$-,+$ & $5$ & $5$ & $5$\\
$+,-,+$ & $6$--$38$ & $6$--$39$ & $6$--$38$\\
$+,-$ & $39$--$66$ & $40$--$66$ & $39$--$64$\\
$+$ & $67$--$120$ & $67$--$120$ & $65$--$120$\\
\midrule
Interior global maximizer & $5$--$49$ & $5$--$50$ & $5$--$49$\\
Right endpoint & $50$--$102$ & $51$--$81$ & $50$--$76$\\
Left endpoint & $103$--$120$ & $82$--$120$ & $77$--$120$\\
\bottomrule
\end{tabular}
\end{center}
The shape regimes are strikingly similar, even though the transition
degrees are sensitive to the envelope, especially the final endpoint
switch. On the consecutive interior-maximizer ranges, all centered
second differences of the detected maximizing locations are positive;
the smallest are about $0.000725$, $0.001110$, and $0.001364$,
respectively. This is numerical evidence of a shared geometry, not a
proof of convexity of any continuously selected stationary branch.

\begin{figure}[htbp]
\centering
\includegraphics[width=.94\textwidth]{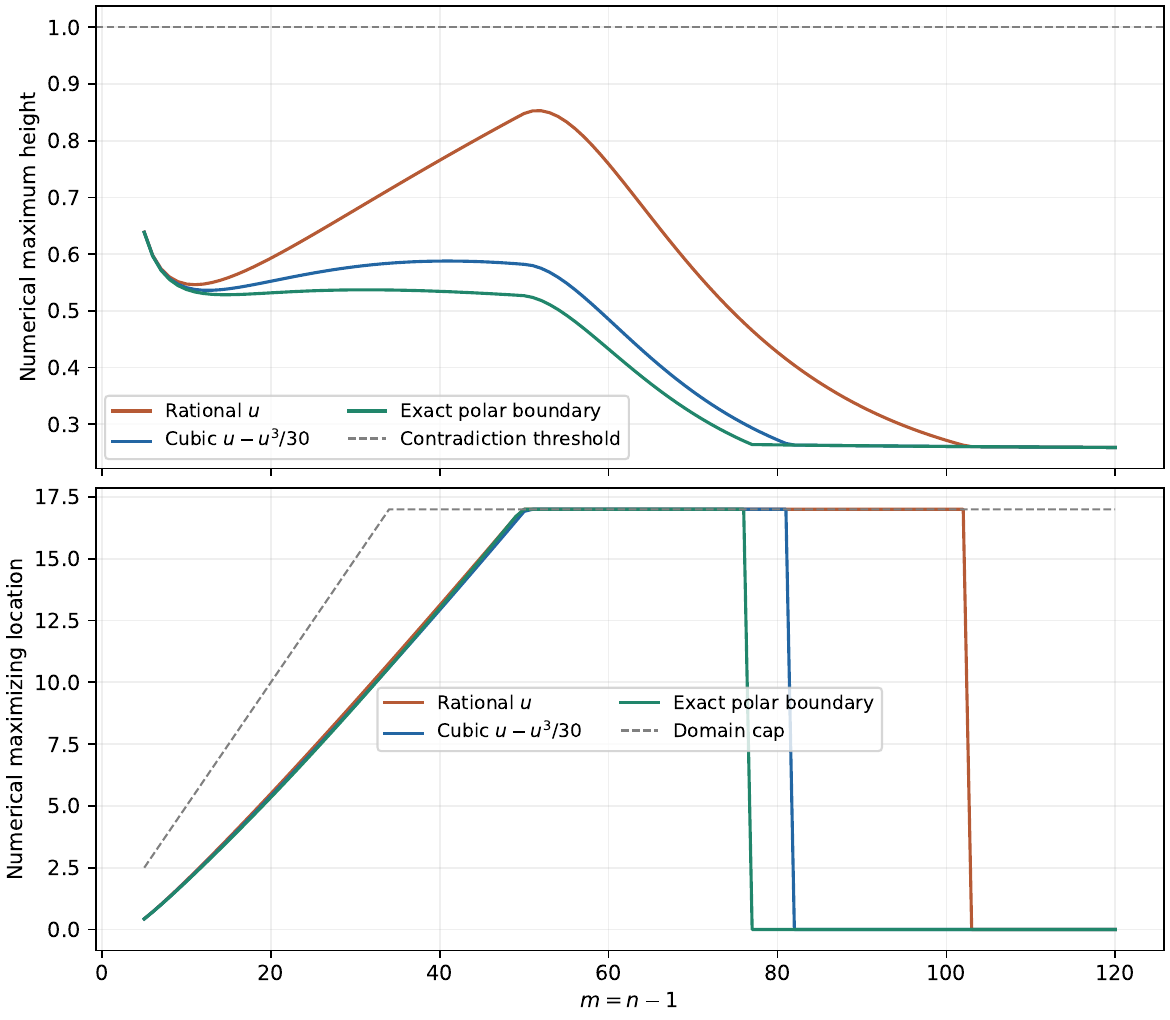}
\caption{Numerically detected maximum heights and maximizing locations
under three polar envelopes on the same capped domain. The lines connect
integer-degree samples. The stronger geometric certificates for the cubic
profile apply only on their explicitly stated range.}
\label{fig:three-envelope-maxima}
\end{figure}

\begin{figure}[htbp]
\centering
\includegraphics[width=.98\textwidth]{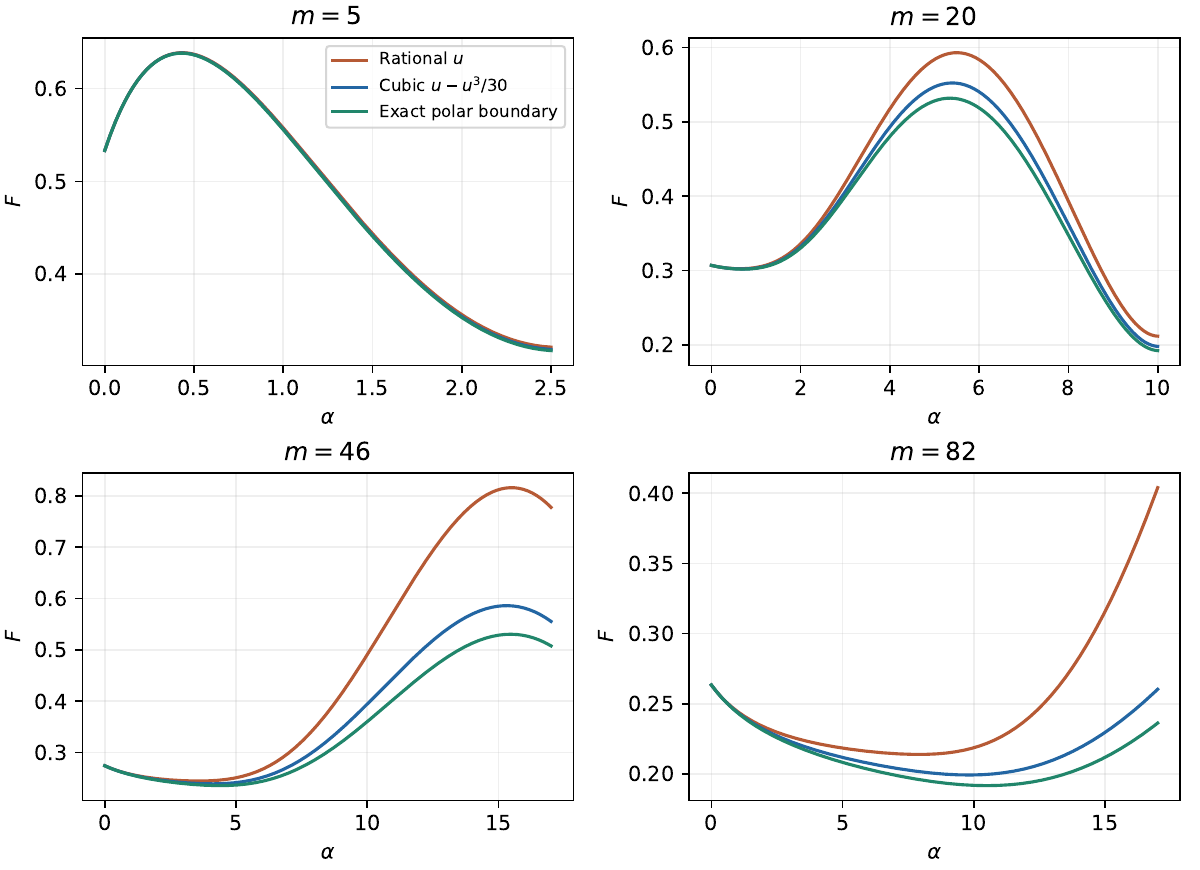}
\caption{Representative profiles for the rational, cubic, and exact polar
boundary envelopes. Pointwise ordering follows analytically from
\eqref{eq:three-profile-order}; the plotted coordinates and heights are
numerical.}
\label{fig:three-envelope-profiles}
\end{figure}

The highest detected maximum of the rational profile over this scan is
$0.852849$ at $(m,\alpha)=(52,17)$, whereas the highest detected maxima
of the cubic and exact-boundary profiles occur at $m=5$, with heights
$0.638422$ and $0.638234$. Selected middle-degree values reveal the
amplification of the envelope correction by the kernel power:
\begin{center}\small
\begin{tabular}{@{}rrrr@{}}
\toprule
$m$ & Rational maximum & Cubic maximum & Polar-boundary maximum\\
\midrule
5 & 0.638798 & 0.638422 & 0.638234\\
20 & 0.592896 & 0.552201 & 0.531739\\
40 & 0.765497 & 0.587764 & 0.534409\\
46 & 0.815764 & 0.585926 & 0.530260\\
52 & 0.852849 & 0.575102 & 0.518391\\
70 & 0.573941 & 0.357602 & 0.319056\\
82 & 0.404036 & 0.263364 & 0.263364\\
\bottomrule
\end{tabular}
\end{center}
All table entries are rounded numerical estimates. They illustrate why
a modest pointwise sharpening of $\eta$ is valuable for global scalar
majorization. In particular, the maximizing location need not be monotone
in envelope height: at $m=46$ the exact-boundary maximizing location
is approximately $15.45636$, between those for the cubic ($15.30804$)
and rational ($15.49628$) profiles. The near-affine interior branch
persists under all three choices, but its detailed position and the
endpoint transition depend on the envelope. This separates robust
features of the scalar reduction from features specific to a particular
majorant.

\clearpage
\section{Integral-free majorants and a shorter analytic completion}
\label{sec:analytic-uniform}

A hand-checkable analytic proof of Sendov's conjecture was already
presented in the author's September 2026 paper~\cite{anhand2026}, publicly
released on 6 September 2026. Starting from the Mazur--Tao scalar reduction,
that proof combined the smooth cubic polar envelope with a two-chord
majorant, monotone-factor estimates, and rational isolation of critical
points whose defining polynomials have degrees at most $7$, $7$, and $4$.
A printed rational table treated the finite degree range, and an analytic
monotonicity argument treated the infinite tail; the cited reduction and
established low-degree cases completed the polynomial theorem. The final
verification consists of analytic arguments and explicitly specified integer
and rational comparisons. Computation supported discovery and auditing,
while all data needed for the finite checks were printed in the paper.

The purpose of the present section is to simplify this earlier analytic
treatment using the geometric analysis in Section~\ref{sec:geometry} and
the weighted estimates in Zhang's recent work~\cite[Section 4]{zhang2026}.
The landscape analysis reveals substantial room below the contradiction
threshold and guides the construction of simpler majorants. We retain
Zhang's small-parameter estimate for $\alpha\le3$ and introduce a split
fourth-power weighted estimate for larger $\alpha$, controlling the
\emph{same} cubic-envelope profile throughout. This estimate retains an
additional inverse power of the degree and reduces the final
large-parameter comparison to a concave quadratic. The resulting uniform
bound $F(m,\alpha)<95/96$ holds for every real $m\ge5$ on
$0\le\alpha\le\min\{m/2,17\}$, replacing the earlier finite degree table
by a short chain of elementary inequalities. The geometric analysis guides
the simplification; the analytic proof below is logically independent of
its computational certificates.

Throughout this section write
\[
 d=1-2\alpha/m,\qquad \delta=1-\ec,\qquad p=(m-3)/2,
 \qquad K=\tfrac14 m(m^2-1)rd^2.
\]
The identity $1-B(t)=t[\delta+d(1-t)]$ will be used repeatedly.

\subsection{Two complementary integral eliminations}
\label{subsec:new-majorants}

The first majorant is our earlier Jensen--Bernoulli estimate:
\begin{equation}\label{eq:bernoulli-retained}
 F\le \mathcal B_m:=E+\frac{6K}{\prod_{j=1}^4[1+d(p+j)]}
                  +\frac K5\ec^p.
\end{equation}
Indeed $B(t)=(1-t)(1-dt)+t\ec$; convexity of $v^p$ yields
$B^p\le(1-t)(1-dt)^p+t\ec^p$. For $d>0$, Bernoulli's inequality gives
$1-t\le(1-dt)^{1/d}$, and integrating the nonnegative affine power up to
$1/d$ gives the first rational term of \eqref{eq:bernoulli-retained}.
The second term follows from $\int_0^1t^4\,dt=1/5$. At $d=0$, $K=0$.
Our preceding full analytic proof of $\mathcal B_m<499/500$ is preserved
in the source supplement. The argument below controls $F$ directly and
makes no use of that longer estimate.

Following Zhang's weighted-maximum idea, define for $k>0$
\[
 M_{p,k}=\max_{0\le v\le1}v^p(1-v)^k
       =\left(\frac p{p+k}\right)^p\left(\frac k{p+k}\right)^k.
\]
The kernel-deficit identity gives
\begin{equation}\label{eq:weighted-k}
 t^3 B(t)^p\le \frac{M_{p,k}\,t^{3-k}}{[\delta+d(1-t)]^k}
 \qquad(0<t\le1).
\end{equation}
In particular $k=3$ gives the explicit bound
\begin{equation}\label{eq:weighted-three}
 W_3=E+K M_{p,3}\frac{2\delta+d}{2\delta^2(\delta+d)^2}.
\end{equation}
It is not uniformly below one: substitution of $m=47$, $\alpha=17/2$
into the rational expression gives $W_3>107/100$. For fixed $\alpha$,
$K M_{p,3}$ tends to $54r e^{-3}$ as $m\to\infty$, whereas the actual
integral contribution decays. This suggests retaining an additional
power of the deficit in the large-degree estimate.

\begin{lemma}[Split fourth-power estimate]\label{lem:split-four}
Set $y=d/(2\delta)$. On the stated domain,
\begin{align}
 F&\le W_4:=E+\frac{24mr d^2}{(m+3)(m+5)(\delta+d/2)^4}
       +K M_{p,4}\frac{9+7y+2y^2}{12\delta^4(1+y)^3}
       \label{eq:weighted-four}\\
 &\le E+\frac{r}{m\delta^2}G(y),\qquad
 G(y):=\frac{y^2(459+304y+171y^2+38y^3)}{3(1+y)^4}.
 \label{eq:weighted-rational}
\end{align}
\end{lemma}
\begin{proof}
On $0\le t\le1/2$, use $B(t)\le1-(\delta+d/2)t$, extend the affine-power
integral to its zero, and use
$\int_0^1t^3(1-t)^p\,dt=6/\prod_{j=1}^4(p+j)$.
Cancellation with $K$ gives the first term after $E$ in
\eqref{eq:weighted-four}. On the second half, \eqref{eq:weighted-k} with
$k=4$ and $1/t\le1+2(1-t)$ yield the other term, since
\[
 \int_{1/2}^1\frac{1+2(1-t)}{[\delta+d(1-t)]^4}\,dt
 =\frac{9+7y+2y^2}{12\delta^4(1+y)^3}.
\]
The identity follows by setting $v=2(1-t)$ and integrating
$(1+v)/(1+yv)^4$; it also holds at $y=0$ by continuity.

For the degree coefficient let $N=m+5\ge10$. The elementary inequality
$\log(1-z)\le-2z/(2-z)$, obtained by differentiation, gives
\begin{align*}
 \log\left(\frac{m^2(m^2-1)M_{p,4}}{4096}\right)
 &\le4\log(1-5/N)+\tfrac{N-8}{2}\log(1-8/N)\\
 &\le-4+\frac{16}{N-4}-\frac{40}{2N-5}\le-4.
\end{align*}
Here the last inequality is precisely $N\ge10$. Since
$e>163/60>19/7$ and $19^5>1024\cdot7^4$, it follows that
\[
 \frac{m(m^2-1)}4 M_{p,4}\le\frac{1024e^{-4}}m<\frac{19}m.
\]
Relax the first degree coefficient using $m/[(m+3)(m+5)]<1/m$.
Substituting $d=2\delta y$ and collecting the two rational terms proves
\eqref{eq:weighted-rational}. The case $d=0$ is immediate from $K=0$.
\end{proof}

\begin{figure}[htbp]
\centering
\includegraphics[width=.98\textwidth]{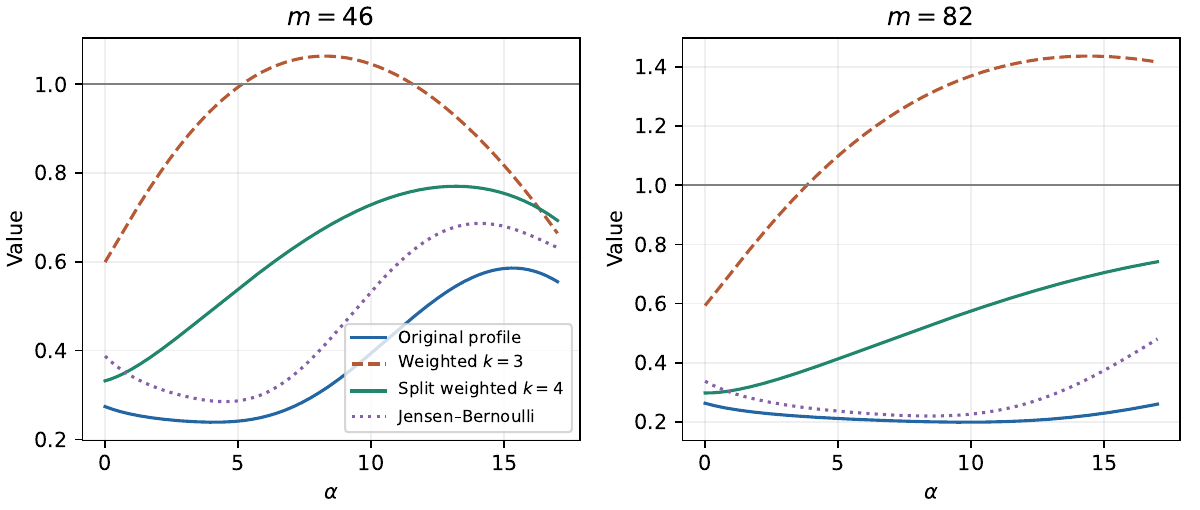}
\caption{Numerical evaluations of the original cubic profile and three
explicit majorants. The unsplit weighted estimate $W_3$ loses too much
information; the split fourth-power estimate $W_4$ retains the degree
decay. The curve for $\mathcal B_m$ is the earlier Jensen--Bernoulli
majorant. The formulas defining all three upper bounds are analytic.}
\label{fig:weighted-comparison}
\end{figure}

\subsection{The range \texorpdfstring{$3\le\alpha\le17$}{3 <= alpha <= 17}}

\begin{lemma}[A rational coefficient estimate]\label{lem:weighted-coefficient}
For $1/2\le u\le17/20$ and $\delta=1-u+u^3/30$,
\[
 \frac{(1-u^2/30)(1-u)^2}{18u\delta^2}
 \le \frac1{40}+\frac{\delta}{6}.
\]
\end{lemma}
\begin{proof}
Since $\delta^2\ge(1-u)(1-u+u^3/15)$, it suffices to replace the left-hand
side by $(1-u^2/30)(1-u)/[18u(1-u+u^3/15)]$. The difference between the
right-hand side and this latter expression equals
\[
 \frac{P(u)}{360u(u^3-15u+15)},
\]
where
\[
 P(u)=2u^7-90u^5+99u^4+890u^3-1925u^2+1335u-300.
\]
The denominator is positive. Put $z=(u-1/2)/(1-u)\ge0$. Expansion gives
\begin{align*}
 64(1+z)^7 P\left(\frac{1/2+z}{1+z}\right)
 &=704z^7+1248z^6+400z^5+4272z^4\\
 &\quad+12284z^3+11380z^2+3570z+57>0.
\end{align*}
This proves the asserted inequality without root isolation.
\end{proof}

\begin{lemma}[An affine bound for $G$]\label{lem:weighted-G}
For $y\ge0$, $G(y)<16y+13/3$.
\end{lemma}
\begin{proof}
Multiply the difference by $3(1+y)^4$. The numerator is
\[
 L(y)=10y^5+34y^4+36y^3-189y^2+100y+13.
\]
For $0\le y\le1/2$, $100y-189y^2\ge(11/2)y$, hence $L>0$.
For $1/2\le y\le1$, set $z=1-y$; expansion gives
\[
 L(1-z)=4-16z+223z^2-272z^3+84z^4-10z^5.
\]
As $0\le z\le1/2$, this is at least $4-16z+87z^2+79z^4>0$;
the quadratic has negative discriminant. For $y\ge1$, put $z=y-1$.
Then $L(1+z)=4+16z+223z^2+272z^3+84z^4+10z^5>0$.
\end{proof}

\begin{proposition}[Uniform estimate in the large-parameter range]
\label{prop:weighted-large-alpha}
If $3\le\alpha\le\min\{m/2,17\}$ and $m\ge5$ is real, then
$F(m,\alpha)<95/96$.
\end{proposition}
\begin{proof}
Here $1/2\le u\le17/20$ and $1/6<\delta<51/100$, since $\delta$ decreases
with $u$ and its endpoint values are $40913/240000$ and $121/240$.
Using $r\le(1-u)/3$, $r/(1-\ec)\le1/3$ and
$1/m=(1-d)(1-u)/(6u)$, we obtain
\[
 E\le \frac16+\frac{1-u}{12}
   +\frac{(7-u)(1-u)}{72u}(1-d)
 \le\frac5{24}+\frac{\delta}{5}(1-d).
\]
The last step uses $(7-u)/(72u)\le13/72<1/5$.
Also,
\[
 \frac{r}{m\delta^2}=\frac{(1-u^2/30)(1-u)^2}{18u\delta^2}(1-d).
\]
Lemmas~\ref{lem:split-four}--\ref{lem:weighted-G}, with $d=2\delta y$,
therefore give
\[
 F\le\frac5{24}+(1-2\delta y)
       \left[\frac{\delta}{5}+\left(\frac1{40}+\frac\delta6\right)
                                   (16y+13/3)\right].
\]
This is a concave quadratic in $y$ at each fixed $\delta$. Write its
nonconstant-budget term as $(1-2\delta y)(A+Dy)$, where
\[
 A=\frac{332\delta+39}{360},\qquad D=\frac{2(20\delta+3)}{15}>0.
\]
Completing the square gives
\[
 (1-2\delta y)(A+Dy)\le\frac{(D+2\delta A)^2}{8\delta D}.
\]
The right-hand side is less than $25/32$ precisely when
$(25/4)\delta D-(D+2\delta A)^2>0$. The latter difference is $Q(\delta)/32400$, with
\[
 Q(s)=-110224s^4-344616s^3+222831s^2+6264s-5184.
\]
To check its sign on $[1/6,51/100]$, note $Q'''(s)<0$ for $s\ge0$ and
$Q'(1/6)=1344109/27>0$. Thus $Q''$ is strictly decreasing; $Q'$ is either
decreasing throughout, or first increasing and then decreasing. Starting
positive, $Q'$ can change sign at most once, from positive to negative.
The minimum of $Q$ on this interval is consequently at an endpoint, and
\[
 Q(1/6)=119639/324>0,\qquad Q(51/100)=1093144896/390625>0.
\]
Hence $F<5/24+25/32=95/96$. Every argument is valid for real $m$.
\end{proof}

\subsection{The range \texorpdfstring{$0\le\alpha\le3$}{0 <= alpha <= 3} and assembly}

For completeness we record the small-parameter estimate from Zhang~\cite[Lemma 4.1]{zhang2026}
in our notation. The following proof uses the same constants and ideas as his
lemma; writing $m=n-1$ makes explicit that no integer-degree step is needed.
The rational polar envelope $u$ suffices in this part.

\begin{lemma}[Zhang's small-parameter estimate, real-parameter form]
\label{lem:weighted-small-alpha}
For every real $m\ge5$ and $0<\alpha\le\min\{3,m/2\}$,
$H_m(\alpha,u)<95/96$.
\end{lemma}
\begin{proof}
Put $n=m+1$, $\delta_0=3/(3+\alpha)\ge1/2$, and
$C_n=n(n-1)(n-2)M_{(n-4)/2,3}$. Then
\[
 C_n=\frac{216n(n-1)(n-2)}{(n+2)^3}
           \left(\frac{n-4}{n+2}\right)^{(n-4)/2}.
\]
The logarithmic derivative is $1/n+1/(n-1)+1/(n-2)
-\tfrac12\int_{n-4}^{n+2}dv/v<0$: split the integral into three
intervals of length two and apply strict convexity of $1/v$, first on
each interval, then to the resulting outer midpoint values.
Thus $C_n\le C_6=405/32<13$, and $C_n\le C_{10}=45/4$ for $n\ge10$.

The elementary part of $H_m(\alpha,u)$ is
$E_u=1/6+1/[4(3+\alpha)]+1/(2m)+1/[4m(3+\alpha)]$.
For $5\le m\le9$, integrate \eqref{eq:weighted-k} with $k=3$ and
$\delta_0$ in place of $\delta$. The integral contribution is at most
\[
 \frac{C_n(3+\alpha)d}{72}\left[1-(1+d/\delta_0)^{-2}\right]
 \le\frac{C_n(3+\alpha)d}{81},
\]
using $d/\delta_0\le2$. Moreover
$(3+\alpha)d\le(3+\alpha)(1-2\alpha/9)
=25/8-(2/9)(\alpha-3/4)^2\le25/8$ and $E_u<3/8$.
Hence $H_m(\alpha,u)<3/8+13\cdot25/(81\cdot8)=71/81<95/96$.

For $m\ge9$, split the integral at $1/2$. On the first half,
the affine-power estimate and $(\delta_0+d/2)^2\ge2\delta_0d$ give
\[
 T_0\le\frac{2(3+\alpha)m}{3(m+3)(m+5)}
 \le\frac{3+\alpha}{192}+\frac{3(3+\alpha)}{8m}.
\]
For the second inequality use $m/[(m+3)(m+5)]\le1/(m+8)$ and
$1/192+3/(8m)-2/[3(m+8)]=(m-24)^2/[192m(m+8)]$.
On the second half, \eqref{eq:weighted-k} with $k=3$ gives
\[
 T_1\le \frac{C_n(3+\alpha)d}{72}
       [1-(1+d/(2\delta_0))^{-2}]
 \le\frac{15(3+\alpha)d}{128}.
\]
The sum $E_u+T_0+T_1$ after these relaxations is affine in $v=1/m$
for fixed $\alpha$, with $d=1-2\alpha v$. Its endpoint at $v=0$ is at
most $1/4+1/32+90/128=63/64$. Its endpoint at $v=1/9$ is strictly
less than $1/3+1/32+1/4+3/8=95/96$, using $E_u\le17/54$ and
$(3+\alpha)(1-2\alpha/9)\le25/8$. This proves the estimate.
\end{proof}

\begin{theorem}[Uniform analytic scalar bound]\label{thm:analytic-uniform}
For every real $m\ge5$ and $0\le\alpha\le\min\{m/2,17\}$,
\[
 \boxed{F(m,\alpha)<\frac{95}{96}<1.}
\]
\end{theorem}
\begin{proof}
For $0<\alpha\le3$, monotonicity in $\eta$ gives
$F(m,\alpha)<H_m(\alpha,u)<95/96$ by
Lemma~\ref{lem:weighted-small-alpha}. For $\alpha\ge3$, apply
Proposition~\ref{prop:weighted-large-alpha}. At zero use the exact
continuous endpoint value $F(m,0)\le8/15$. At $d=0$, the integral
contribution is zero and the same estimates remain valid.
\end{proof}

This analytic completion is independent of the geometric certificates.
Its smaller ceiling for $F$ does not assert the same ceiling for the
larger Jensen--Bernoulli majorant $\mathcal B_m$. The earlier theorem
about that particular majorant remains a complementary result, retained
with its proof in the source supplement.

\Needspace{10\baselineskip}
\section{Completion, strictness, and the equality case}\label{sec:completion}

\begin{proposition}[Strict interior conclusion]\label{prop:strict-interior}
Under the hypotheses of Theorem~\ref{thm:sendov}, if $|a|<1$, then
$d_P(a)<1$.
\end{proposition}
\begin{proof}
Suppose every critical distance from an interior zero were at least one.
A multiple zero is already critical. At $a=0$, the product identity in
Appendix~\ref{app:reduction-complete} gives an immediate contradiction.
For $2\le n\le5$, Appendix~\ref{sec:low-polynomial-degrees} rules out the
same weak separation. Otherwise $m=n-1\ge5$, and the reduction gives
\eqref{eq:reduction-summary}. The envelope theorem and strict increase
in $\eta$ then yield
\[
 1\le H_m(\alpha,\eta)
 <F(m,\alpha)<\frac{95}{96}<1,
\]
where the last bound is Theorem~\ref{thm:analytic-uniform}.
This contradiction proves the claim. The argument requires only that
the distinguished zero be interior; other zeros may lie on the circle.
\end{proof}

The boundary case and its equality characterization are classical;
see Rubinstein~\cite[Theorem 1]{rubinstein1968}. The proof below uses the
Meir--Sharma identity~\cite{meirsharma1969}. The short argument
below is recorded in Tao's exposition~\cite[Theorem 14 and Lemma 15]{tao2026},
which attributes its underlying observation to Tang and
Zhang~\cite[Remark 5.1]{tangzhang2025}.

\begin{lemma}[Boundary equality]\label{lem:boundary-equality}
If $P(1)=0$, all zeros lie in the closed unit disk, and every critical
point satisfies $|1-\xi_j|\ge1$, then $P(z)=c(z^n-1)$ for some $c\ne0$.
\end{lemma}
\begin{proof}
The zero at one is simple. Write the other zeros as $z_1,\ldots,z_m$ and
put $q_j=1/(1-\xi_j)$, so $|q_j|\le1$.
Logarithmic differentiation of $P'$ and of $P(z)/(z-1)$ at one gives
\[
 \sum_{j=1}^m q_j=\frac{P''(1)}{P'(1)}
                =2\sum_{j=1}^m\frac1{1-z_j}.
\]
For $|z|\le1$, $z\ne1$,
\[
 \operatorname{Re}\frac1{1-z}-\frac12
       =\frac{1-|z|^2}{2|1-z|^2}\ge0.
\]
Taking real parts therefore gives $m\ge\sum\operatorname{Re}q_j\ge m$.
Thus every $\operatorname{Re}q_j=1$; with $|q_j|\le1$ this forces $q_j=1$.
All critical points are zero, so $P'=cnz^{n-1}$. Integration and $P(1)=0$
give the asserted form.
\end{proof}

The resulting extremal characterization is the one proposed by Phelps
and Rodriguez~\cite{phelpsrodriguez1972}.

\begin{theorem}[Phelps--Rodriguez form]\label{thm:strict-equality}
Let $P$ have degree $n\ge2$ and all zeros in the closed unit disk.
For every zero $a$, one has $d_P(a)<1$, unless
\[
 |a|=1\quad\hbox{and}\quad P(z)=c(z^n-a^n)\quad(c\ne0).
\]
In that exceptional case $d_P(a)=1$.
\end{theorem}
\begin{proof}
The interior case is Proposition~\ref{prop:strict-interior}.
For a boundary zero, rotate to $a=1$ and apply
Lemma~\ref{lem:boundary-equality}. Reversing the rotation gives the
exceptional form, whose derivative has only the zero at the origin.
\end{proof}

This proves Theorem~\ref{thm:sendov} as well. The strictness and equality
form are known strengthenings of the same theorem, included for
completeness rather than claimed as new consequences of a smaller
auxiliary ceiling.

\subsection{What the scalar gap does and does not measure}

The analytic ceiling $95/96$ and the stronger geometric ceiling
$16/25$ bound $F$, an auxiliary obstruction after several inequalities.
Neither is a critical-distance estimate. For $P(z)=z^n-1$, all critical points are zero and each root on
the unit circle has nearest critical distance one. Thus Sendov's universal
distance constant is sharp.

There is also no positive degree-independent constant $c$ for which
\[
 d_P(a)\le1-c(1-|a|)\qquad\hbox{or}\qquad
 d_P(a)\le1-c(1-|a|^2)
\]
holds for all such polynomials. Indeed, take $P_n(z)=z^n-z$ and $a=0$.
Its zeros are zero and the $(n-1)$st roots of unity, while
\[
 P_n'(z)=nz^{n-1}-1,\qquad
 d_{P_n}(0)=n^{-1/(n-1)}\longrightarrow1.
\]
This example rules out a uniform positive radial margin even at the center.

A degree-dependent center bound is both explicit and sharp. If $P(0)=0$
and its other zeros are $z_1,\ldots,z_{n-1}$, the product identity gives
\[
 d_P(0)^{n-1}\le\prod_{j=1}^{n-1}|\xi_j|
 =\frac1n\prod_{j=1}^{n-1}|z_j|\le\frac1n.
\]
Hence $d_P(0)\le n^{-1/(n-1)}$, with equality for $P_n$ above.

\begin{corollary}[A qualitative fixed-degree interior margin]
For each fixed $n\ge2$ and $0\le r<1$, there exists
$\varepsilon(n,r)>0$ such that
\[
 |a|\le r\quad\Longrightarrow\quad
 d_P(a)\le1-\varepsilon(n,r)
\]
for every degree-$n$ polynomial whose zeros lie in the closed unit disk.
\end{corollary}
\begin{proof}
After dividing by the leading coefficient, parameterize the family by
the distinguished root $a$ and the other $n-1$ roots. The parameter set
is compact. Continuity of polynomial roots with respect to coefficients
implies continuity of $d_P(a)$. Its maximum is attained and is strictly
less than one by Proposition~\ref{prop:strict-interior}.
\end{proof}
This compactness argument gives no numerical value of $\varepsilon$.
Extracting an explicit degree-dependent margin from a distance-parameter
version of the reduction is a separate problem.

\section{Discussion and comparison of approaches}\label{sec:comparison}

The primary object of this paper is the geometry of the fixed
cubic-envelope profile. Zhang's recent analytic proof~\cite{zhang2026}
is shorter as a proof of Sendov itself: it switches from the refined to
the original origin inequality at $\alpha=3$, eliminating the need for the
cap and the strengthened cubic envelope. His matrix formulation uses
the $D$-companion matrices introduced by Cheung and Ng~\cite{cheungng2006}
and recovers
the same scalar origin and reflection inequalities. Our exact
correspondence is that the right-hand side of his equation (3.9) is
$H_m(\alpha,u)$, with $m=n-1$. His Lemma 4.1 is the small-parameter
part of our shortened completion; we attribute it explicitly. The
large-parameter fourth-power estimate of Section~\ref{sec:analytic-uniform}
maintains the single refined profile, rather than using his Lemma 4.2
for a different scalar obstruction.

\begin{center}\small
\begin{tabularx}{\textwidth}{@{}>{\raggedright\arraybackslash}p{.23\textwidth}>{\raggedright\arraybackslash}p{.26\textwidth}>{\raggedright\arraybackslash}X@{}}
\toprule
Treatment & Structure & Scope\\
\midrule
Preceding arXiv version & 34 intervals, 272 nodal maxima, 15 first-interval rows & Piecewise envelope, degree log-concavity; ceiling $97/100$.\\[3pt]
Two-chord, forty-interval completion & 120 degree maxima, 232 adjacent comparisons & Cubic envelope; one rational certificate across integer degrees, ceiling $1999/2000$.\\[3pt]
Geometric tangent proof & Complete curvature covers through $m=82$; certified evaluations & Cubic envelope; maximizing-point classification and smooth anchor; $F<16/25$ for all integer degrees.\\[3pt]
Earlier analytic completion & Jensen--Bernoulli integral elimination and separate rational/powered budgets & $F\le\mathcal B_m<499/500$ for real $m\ge5$ on the capped domain.\\[3pt]
Zhang's analytic proof & Weighted kernel-deficit estimates; switch of inequalities at $\alpha=3$ & Proof of Sendov without cap or cubic envelope.\\[3pt]
Present analytic completion & Weighted third power for small $\alpha$; split fourth power and affine majorant for large $\alpha$ & $F<95/96$ for real $m\ge5$ on the capped domain; no geometric certificate required.\\
\bottomrule
\end{tabularx}
\end{center}

The numbers in the table need not compare like with like. The rational
and cubic polar envelopes define different profiles; the earlier
$499/500$ ceiling also bounds the explicit larger majorant $\mathcal B_m$.
The new $95/96$ ceiling bounds the cubic profile and does not assert a
new bound for $\mathcal B_m$ itself. The stronger $16/25$ bound is supported
by the geometric certificates and an analytic tail. Its purpose is to
quantify the ample slack available for the simpler completion.

The three-envelope comparison in Section~\ref{subsec:three-envelope-study}
shows why the geometric classification is not simply a property of a
particular curve fit. The ordering and local agreement of the profiles
are analytic; the persistence of their qualitative curvature regimes and
positive discrete second differences is numerical outside the certified
cubic range. It is especially important that the location of the
maximizer is not ordered by envelope height. An analytic derivation of
the curvature and convexity of the actual continuously varying maximizing
branch, rather than of the prescribed rule $\ell$, remains open in this
work. Other questions include sharper scalar bounds for the exact polar
boundary and the behavior of the landscape without the cap. None is
needed for the present polynomial conclusion.

The central methodological point is that numerical geometric exploration
and the design of elementary majorants are complementary. Accurate
optimization is not necessary to obtain a contradiction, but it exposes
where slack can be spent and where a small envelope correction matters.
A less tight analytic majorant may be far simpler to verify. The original
computer-assisted results, exact certificates, and the previous analytic
proof of the Bernoulli bound are therefore retained as distinct supporting
results, not as competing claims of priority for Sendov's theorem.

\section{Verification and reproducibility}\label{sec:reproducibility}

The supplementary proofs, verification programs, certificates, and numerical
data are collected in a separate complete research archive for this paper.
The companion project repository is
\url{https://github.com/SJanPerth/Sendov-Proof-Distillation}.
All commands below are run from the extracted root of the complete archive.

The analytic completion in Section~\ref{sec:analytic-uniform} is specified
by the written inequalities, elementary integrals, and explicit polynomial
sign proofs. The optional exact algebra check
\begin{verbatim}
python audit/check_weighted_completion.py
\end{verbatim}
verifies the new rational identities, the quartic endpoint values, the
unsplit-bound counterexample, and the local expansion coefficients. It
is supplementary; no parameter grid or numerical root isolation enters
the proof.

The geometric classification has a different verification requirement.
Odd-degree curvature statements use exact rational polynomials and
complete Bernstein sign covers. Even-degree statements use outward-rounded
dyadic intervals, explicit integration remainders, and complete covers.
Appendix~\ref{app:geometric-methods} explains the certificate rules and the
one-sided treatment near zero at $m=6$. These methods are retained
unchanged. For stored certificates and source fingerprints, run
\begin{verbatim}
python verify_revision.py --check-stored
\end{verbatim}
The full reconstruction command is
\begin{verbatim}
python verify_revision.py --rebuild-geometry --jobs 2
\end{verbatim}
The first command checks stored records; only the second reconstructs
the curvature covers. The archive retains the earlier reconstruction
receipts. The older Jensen--Bernoulli analytic proof and its optional
algebra-checking programs are also retained; the command
\texttt{python verify\_revision.py --analytic} checks that preceding route.

The new comparison of three polar envelopes can be reproduced by
\begin{verbatim}
python geometry/envelope_study/compare_envelopes.py
\end{verbatim}
The script uses NumPy, SciPy and Matplotlib. The two numerical CSV tables,
figures, and comparison summary are included. Agreement between the two
quadrature resolutions is a numerical diagnostic, not a guarantee of
exhaustive root isolation. The new study and its plots are logically
independent of both the analytic completion and the retained rigorous
cubic certificates.

\appendix
\section{From a polynomial counterexample to the scalar obstruction}
\label{app:reduction-complete}

This section supplies the polynomial reduction used in the scalar proof.
The source is Tao~\cite[Lemma~6 and Propositions~10--11]{tao2026}
(checked 29 November 2026). The notation below is reconciled directly with
equations (11)--(21) of that source. We give all needed deductions in full and
replace the decimal calculation in its bound on $\alpha$ by explicit
rational comparisons. The product polynomial below is called $G$, to avoid
confusing it with the scalar profile $F$ used elsewhere in this manuscript.

\subsection{Normalization and exact identities}

Let $p$ have degree $n$, all zeros in the closed unit disk, and an
interior distinguished zero whose distances from every critical point are
at least one. Rotate the zero to $a\in[0,1)$ and make $p$ monic.
Its other zeros are $z_1,\ldots,z_m$, where $m=n-1$, and the critical
points $\xi_1,\ldots,\xi_m$ satisfy $|a-\xi_j|\ge1$.
All that is used below is $a<1$ and $|z_j|\le1$; the other zeros need
not lie strictly inside the disk. The distinguished zero is simple,
since a multiple zero would itself be a critical point.

This setting also follows from a strict counterexample to Sendov's
non-strict conclusion: if its least critical distance is $d>1$, scaling
by $1/d$ makes the distinguished zero interior and all its critical
distances at least one. The stronger interior formulation will be useful
for the equality statement in the main text.

The case $a=0$ is impossible. Indeed, evaluating the derivative of
$p(z)=z\prod_{j=1}^m(z-z_j)$ at zero gives
\[
 \prod_{j=1}^m(-z_j)=n\prod_{j=1}^m(-\xi_j).
\]
The modulus on the left is at most one, whereas that on the right is at
least $n$. Hence assume $0<a<1$ and put
\[
 q_j=\frac1{a-\xi_j},\qquad
 \mu=x+iy=\frac1m\sum_{j=1}^m q_j,
 \qquad
 G(t)=\prod_{j=1}^m(1-atq_j),\qquad
 W=\prod_{j=1}^m z_jq_j.
\]
Then $0<|q_j|\le1$, $|z_j|\le1$, $|\mu|\le1$, and $|W|\le1$.
The following four identities are exact:
\begin{align}
 \frac{a+\sum_jz_j}{n}
 &=a-\frac1m\sum_j\frac1{q_j},
 \label{eq:red-centroid}\\
 \prod_{j=1}^m\frac{1-az_j}{a-z_j}
 &=\int_0^1\prod_{j=1}^m\bigl(a+t(1-a^2)q_j\bigr)\,dt,
 \label{eq:red-polar-identity}\\
 \int_0^1G(t)\,dt&=\frac{(-1)^mW}{n},
 \label{eq:red-origin1}\\
 G(1)&=\frac{(-1)^m}{n}
 \left(W+aW\sum_{j=1}^m\frac1{z_j}\right).
 \label{eq:red-origin2}
\end{align}
Whenever a $z_j$ vanishes, $W/z_j$ means the product after cancellation;
thus the last identity is polynomial and has no singularity.

For completeness, the factorizations
\[
 p(z)=(z-a)\prod_j(z-z_j),\qquad
 p'(z)=n\prod_j\left(z-a+\frac1{q_j}\right)
\]
give the centroid identity by comparing the coefficients of $z^{n-2}$ in
$p'$. Since $p(a)=0$, integration on the line segment from $a$ to $z$ gives
\[
 p(z)=(z-a)\int_0^1 n\prod_j\left(t(z-a)+\frac1{q_j}\right)\,dt.
\]
Comparing its value at $z=1/a$ with the zero factorization, and dividing by
$p'(a)(1/a-a)$, gives the polar identity after multiplying by $a^m$.
Comparing the values at zero gives \eqref{eq:red-origin1}; evaluating $p'(0)$
in the two factorizations gives \eqref{eq:red-origin2}. These arguments
also prove the identities when zeros other than $a$ are repeated.

\subsection{The polar constraint}

Define
\[
 \alpha=\frac m2(1-a^2),\qquad
 \eta=1-2ax+a^2,
 \qquad \beta(t)=1-2axt+a^2t^2.
\]
Thus $0<\alpha<m/2$ and
\begin{equation}
 \beta(t)=(1-t)^2+\eta t+\frac{2\alpha}{m}t(1-t).
 \label{eq:red-kernel}
\end{equation}
Moreover,
\[
 \eta=(1-a)^2+2a(1-x)>0,
 \qquad
 \eta=1-x^2+(x-a)^2,
\]
and $\beta(t)>0$ for $0\le t\le1$.

For $|z|\le1$,
\[
 |1-az|^2-|a-z|^2=(1-a^2)(1-|z|^2)\ge0.
\]
Consequently the modulus of the left side of
\eqref{eq:red-polar-identity} is at least one. By the triangle inequality,
AM--GM applied to the squared moduli, and $|q_j|\le1$, this implies
\begin{align}
 1&\le\int_0^1\prod_j|a+t(1-a^2)q_j|\,dt\nonumber\\
 &\le\int_0^1
 \left(a^2+2atx(1-a^2)+t^2(1-a^2)^2\right)^{m/2}\,dt.
 \label{eq:red-raw-polar}
\end{align}
Put $c=1-a^2>0$. The quadratic inside the last integral is at most
\[
 1+c\{-1+(2-\eta)t\},
\]
because the difference is $c^2t(1-t)$. Both the original quadratic and its
affine majorant are nonnegative. Since $1+s\le e^s$ and $mc/2=\alpha$, integration gives the
strict inequality
\begin{equation}
 1<I(\alpha,\eta):=
 \int_0^1e^{\alpha[-1+(2-\eta)t]}\,dt.
 \label{eq:red-polar}
\end{equation}
Strictness follows, for example, from $c^2t(1-t)>0$ for $0<t<1$.
If $\eta\ge1$, the exponential integrand is at most one throughout the
interval and strictly smaller on an interval of positive length. Therefore
$0<\eta<1$. Evaluating the integral yields
\[
 I(\alpha,\eta)=
 \frac{e^{\alpha(1-\eta)}-e^{-\alpha}}{\alpha(2-\eta)}.
\]
It follows that $e^{\alpha(1-\eta)}>\alpha$ and hence
\begin{equation}
 1-\eta>\frac{\log\alpha}{\alpha}.
 \label{eq:red-log-bound}
\end{equation}
Also $1-\eta=2a(x-a/2)>0$, so $x>a/2>0$.

\subsection{Two elementary product estimates}

We give elementary proofs of the two product estimates used to control
the origin identities; neither requires a separate imported theorem.

First, for $m\ge3$ and nonnegative $v_1,\ldots,v_m$,
\begin{equation}
 \frac1m\sum_{j=1}^m\prod_{k\ne j}v_k
 \le\left(\frac1m\sum_{j=1}^m v_j\right)^{m-1}
 \le\left(\frac1m\sum_{j=1}^m v_j^2\right)^{(m-1)/2}.
 \label{eq:red-maclaurin}
\end{equation}
The second inequality is Cauchy--Schwarz. For the first, fix the sum of
the $v_j$. On the resulting compact simplex choose a maximizer of the
elementary symmetric polynomial $e_{m-1}$ that minimizes $\sum_jv_j^2$.
If two entries $v_i,v_j$ are unequal, replacing them by their average
cannot decrease $e_{m-1}$: the terms involving both entries have the form
$v_iv_j e_{m-3}(\text{other entries})$, and the terms involving just one
depend only on $v_i+v_j$. The coefficient $e_{m-3}$ is nonnegative.
The averaging strictly decreases $\sum_jv_j^2$, a contradiction. The
maximizer therefore has all entries equal, which proves the first bound.

Second, for points $w_1,\ldots,w_N$ in the closed unit disk,
\begin{equation}
 \left(\prod_j|w_j|\right)
 \sum_j\left|\frac1{w_j}-\overline{w_j}\right|
 \le1-\prod_j|w_j|^2.
 \label{eq:red-defect}
\end{equation}
The expression on the left is interpreted after cancellation at zero.
Writing $r_j=|w_j|$, it equals
$D_N=\sum_j(1-r_j^2)\prod_{k\ne j}r_k$.
If $P=\prod_{j<N}r_j$ and $r=r_N$, induction gives
\[
 D_N\le r(1-P^2)+P(1-r^2)\le1-P^2r^2,
\]
because the difference in the last comparison is
$(1-P)(1-r)(1-Pr)\ge0$. This argument includes zero radii directly.

\subsection{The raw origin inequality}

Assume $n\ge5$, equivalently $m\ge4$. Differentiating $G$ gives the exact
identity
\[
 G'(t)=-ma\mu G(t)
       -a^2t\sum_{j=1}^m q_j^2\prod_{k\ne j}(1-atq_k).
\]
Since
\[
 \frac1m\sum_j|1-atq_j|^2
 =1-2axt+\frac{a^2t^2}{m}\sum_j|q_j|^2\le\beta(t),
\]
\eqref{eq:red-maclaurin} implies
\[
 |G'(t)+ma\mu G(t)|\le a^2mt\,\beta(t)^{(m-1)/2}.
\]
Integrating, using $G(0)=1$, and applying the triangle inequality gives
\begin{equation}
 1\le\left|G(1)+ma\mu\int_0^1G(t)\,dt\right|
       +a^2m\int_0^1t\beta(t)^{(m-1)/2}\,dt.
 \label{eq:red-integrated-origin}
\end{equation}

Conjugating the centroid identity and multiplying by $W$ yields
\[
 W\sum_j\left(\frac mn\overline{z_j}
                       +\frac1{\overline{q_j}}\right)
 =\frac{am^2W}{n}.
\]
Apply \eqref{eq:red-defect} to the $2m$ points
$z_1,\ldots,z_m,\overline{q_1},\ldots,\overline{q_m}$.
Since $m/n\le1$, replacing the conjugate terms by reciprocals, and vice
versa, gives an error $\Delta$ satisfying
\[
 W\left(\frac mn\sum_j\frac1{z_j}+m\mu-\frac{am^2}{n}\right)
 =\Delta,\qquad |\Delta|\le1-|W|^2.
\]
In particular,
\[
 W\sum_j\frac1{z_j}=maW-n\mu W+\frac nm\Delta.
\]
Substitution in \eqref{eq:red-origin1}--\eqref{eq:red-origin2} gives
\[
 \left|G(1)+ma\mu\int_0^1G(t)\,dt\right|
 \le\rho D+\frac am(1-\rho^2),
 \qquad
 \rho=|W|,\quad D=\left|\frac{1+ma^2-a\mu}{n}\right|.
\]
We next show $D>2a/m$. Indeed,
\[
 D\ge\frac{1+ma^2-a}{n},
\]
and the required strict inequality reduces to
\[
 m^2a^2-(3m+2)a+m>0.
\]
Its leading coefficient is positive and its discriminant is negative:
\[
 4m^3-(3m+2)^2
 =60+(m-4)(4m^2+7m+16)>0\qquad(m\ge4).
\]
It follows that $\rho D+(a/m)(1-\rho^2)$ increases for $0\le\rho\le1$,
and hence is at most $D$.

The complex number inside $D$ has real part
$s=(1+ma^2-ax)/n>ma^2/n>0$ and imaginary part $-ay/n$.
Using $|s+it|\le s+t^2/(2s)$ and $y^2\le1-x^2$,
\[
 D\le\frac{1+ma^2-ax}{n}+\frac{1-x^2}{2mn}.
\]
Insert these estimates into \eqref{eq:red-integrated-origin}, multiply by
$n$, and use $m(1-a^2)=2\alpha$. We obtain
\begin{equation}
 \boxed{2\alpha+ax\le\frac{1-x^2}{2m}
 +a^2n m\int_0^1t\beta(t)^{(m-1)/2}\,dt.}
 \label{eq:red-raw-origin}
\end{equation}
This proves Proposition~11(i) in the notation of the scalar argument.

\subsection{An explicit rational proof that \texorpdfstring{$\alpha\le17$}{alpha<=17}}

Suppose for contradiction that $\alpha>17$. Since $\alpha<m/2$ and $m$
is an integer, $m\ge35$ and $n\ge36$. We already know $x>0$.
For $0\le t\le\min\{x/a,1\}$,
\[
 \beta(t)\le1-axt\le e^{-axt},
\]
whereas on the remaining interval, if present, $\beta(t)\le\eta$.
Consequently
\[
 \int_0^1t\beta(t)^{(m-1)/2}\,dt
 \le\frac4{(m-1)^2a^2x^2}+\frac12\eta^{(m-1)/2}.
\]
Dropping the positive term $ax$ from \eqref{eq:red-raw-origin} and using
$1-x^2\le1$ gives
\[
 2\alpha\le\frac1{2m}
 +\frac{4nm}{(m-1)^2x^2}
 +\frac{a^2nm}{2}\eta^{(m-1)/2}.
\]
Since $x^2\ge1-\eta$ and \eqref{eq:red-log-bound} holds,
\[
 \frac1{x^2}<\frac\alpha{\log\alpha},\qquad
 \eta\le e^{-(1-\eta)}<\alpha^{-1/\alpha}.
\]
Set $v=a^2/(1-a^2)>0$, so $m=2\alpha(1+v)$. Dividing the last origin
estimate by $\alpha$ and simplifying yields
\begin{equation}
 2\le\frac1{2m\alpha}
 +\frac{4nm}{(m-1)^2\log\alpha}
 +\frac{2n}{m}v(1+v)\alpha^{-v}\alpha^{1/(2\alpha)}.
 \label{eq:red-cap-prebound}
\end{equation}

We record three elementary numerical bounds, with exact certificates:
\begin{equation}
 \log17>\frac{14}{5},\qquad
 17^{1/34}<\frac{11}{10},\qquad
 \max_{v\ge0}v(1+v)e^{-14v/5}<\frac3{16}.
 \label{eq:red-cap-constants}
\end{equation}
For the first bound, the exponential series and a geometric tail give
\[
 e^{14/5}\le
 \sum_{k=0}^4\frac{(14/5)^k}{k!}
 +\frac{(14/5)^5}{5!\,[1-(14/5)/6]}
 =\frac{311791}{18750}<17.
\]
For the second, $e^3>\sum_{k=0}^5 3^k/k!=92/5>17$, hence
$\log17<3$. The geometric majorant of the exponential series gives
\[
 17^{1/34}<e^{3/34}\le\frac1{1-3/34}
 =\frac{34}{31}<\frac{11}{10}.
\]
For the third, the sign of the derivative of $v(1+v)e^{-14v/5}$ is the
sign of $\phi(v)=1-(4/5)v-(14/5)v^2$. This polynomial strictly decreases
on $[0,\infty)$ and has one zero $v_*$, with
\[
 \phi(47/100)=\frac{137}{25000}>0,
 \qquad
 \phi(473/1000)=-\frac{12103}{2500000}<0.
\]
The maximum occurs at $v_*$. Because
\[
 \frac{473}{1000}\frac{1473}{1000}<\frac{697}{1000},
 \qquad
 e^{329/250}>
 \sum_{k=0}^{5}\frac{(329/250)^k}{k!}>\frac{3719}{1000},
\]
we obtain
\[
 v_*(1+v_*)e^{-14v_*/5}<\frac{697}{3719}<\frac3{16};
\]
the last comparison is $11152<11157$. The intermediate finite-sum
comparison is also rational; its positive difference is
$75730383899/117187500000000$.

Finally, $nm/(m-1)^2=(m+1)m/(m-1)^2$ and $n/m$ decrease with $m$.
The function $\alpha^{1/(2\alpha)}$ decreases for $\alpha\ge17$,
and $\alpha^{-v}\le17^{-v}<e^{-14v/5}$. Applying
\eqref{eq:red-cap-constants} to \eqref{eq:red-cap-prebound} therefore gives
\[
 2<\frac1{2\cdot35\cdot17}
 +\frac{4\cdot36\cdot35}{34^2}\frac5{14}
 +\frac{72}{35}\frac3{16}\frac{11}{10}
 =\frac{401003}{202300}<2,
\]
a contradiction. Hence $\alpha\le17$. Replacing the first strict sign
by a non-strict sign would also suffice; the final rational gap is
$3597/202300$.

\subsection{The scalar obstruction and exact change of variables}

The identity
\[
 \beta(t)=(1-axt)^2+a^2t^2(1-x^2)
\]
has nonnegative terms because $0<ax<1$ and $|x|\le1$.
For $s=(m-1)/2\ge1$, the mean value theorem gives
\[
 \beta(t)^s\le(1-axt)^{m-1}
 +s a^2t^2(1-x^2)\beta(t)^{s-1}.
\]
Since $1/(ax)>1$, a beta-integral calculation gives
\[
 \int_0^1t(1-axt)^{m-1}\,dt
 \le\int_0^{1/(ax)}t(1-axt)^{m-1}\,dt
 =\frac1{a^2x^2m(m+1)}.
\]
Insert these two inequalities into \eqref{eq:red-raw-origin} to obtain
\[
 2\alpha+ax\le\frac{1-x^2}{2m}+\frac1{x^2}
 +\frac{a^4n m(m-1)(1-x^2)}2
       \int_0^1t^3\beta(t)^{(m-3)/2}\,dt.
\]
Now use
\[
 1-x^2\le\eta,\quad
 \frac1{x^2}\le1+\frac\eta{1-\eta},\quad
 ax=1-\frac\alpha m-\frac\eta2.
\]
After subtracting $ax$ and dividing by $2\alpha>0$,
\begin{align}
 1\le{}&\frac\eta{2\alpha(1-\eta)}+\frac\eta{4\alpha}
       +\frac1{2m}+\frac\eta{4\alpha m}\nonumber\\
 &+\frac{a^4n m(m-1)\eta}{4\alpha}
       \int_0^1t^3\beta(t)^{(m-3)/2}\,dt.
 \label{eq:red-scalar-obstruction}
\end{align}
Finally,
\[
 a^4n m(m-1)
 =\left(1-\frac{2\alpha}{m}\right)^2(m+1)m(m-1)
 =\frac{(m^2-1)(m-2\alpha)^2}{m}.
\]
Together with \eqref{eq:red-kernel}, this is exactly
$1\le H_m(\alpha,\eta)$ with the elementary term, squared prefactor,
and integral exponent used in this manuscript. We have therefore proved
the complete implication
\[
 \text{normalized polynomial counterexample with }n\ge5
 \quad\Longrightarrow\quad
 \begin{cases}
 0<\alpha<m/2,\quad \alpha\le17,\\
 0<\eta<1,\quad I(\alpha,\eta)>1,\\
 1\le H_m(\alpha,\eta).
 \end{cases}
\]
No converse is required, and no assertion that arbitrary admissible scalar
pairs come from polynomials is used. The low-degree closure and the
upper bound for this scalar obstruction are separate components of the
full proof.

\section{An elementary closure for polynomial degrees two through five}
\label{sec:low-polynomial-degrees}

The cases of degrees three through five are classical
\cite{brannan1968,rubinstein1968,meirsharma1969}; the degree-two case is
immediate. For completeness, we give the elementary fourth-moment
argument, also used by Zhang~\cite[Lemma 2.2]{zhang2026}, that closes all
four degrees $2\le n\le5$ within the normalization and polar identity of
Appendix~\ref{app:reduction-complete}. This treatment requires no scalar
curvature certificate.

\begin{proposition}\label{prop:low-polynomial-degrees}
For $n=2,3,4,5$, a polynomial whose zeros lie in the closed unit disk
has a critical point at distance strictly less than one from each of its
interior zeros. In particular, Sendov's non-strict conclusion holds in
these degrees.
\end{proposition}
\begin{proof}
Suppose an interior distinguished zero has all critical distances at
least one, and use the normalization in
Section~\ref{app:reduction-complete}. The case $a=0$ has already been
excluded. Thus $0<a<1$, $1\le m=n-1\le4$, and $|q_j|\le1$.
The polar identity and disk inequality give
\[
 1\le\int_0^1\prod_{j=1}^m
       |a+(1-a^2)tq_j|\,dt
 \le\int_0^1\bigl(a+(1-a^2)t\bigr)^m\,dt.
\]
Put $Q(t)=a+(1-a^2)t\ge0$.  On the probability space $[0,1]$,
Jensen's inequality for the concave function $s\mapsto s^{m/4}$ gives
\[
 \int_0^1Q(t)^m\,dt
 =\int_0^1\bigl(Q(t)^4\bigr)^{m/4}\,dt
 \le\left(\int_0^1Q(t)^4\,dt\right)^{m/4}.
\]
Here $m/4\in(0,1]$, including equality at $m=4$.
Direct polynomial integration and factorization give
\begin{align*}
 \int_0^1Q(t)^4\,dt
 &=1-\frac{(1-a)^3(1+a)}5
           (a^4-3a^3+3a+4),\\
 a^4-3a^3+3a+4&=a^4+3a(1-a^2)+4>0
 \qquad(0\le a<1).
\end{align*}
Thus $\int_0^1Q^4<1$, and the preceding estimates imply $1<1$,
a contradiction. A strict violation of Sendov at a boundary zero scales
to this interior setting, so the non-strict assertion follows as well.
\end{proof}

The proof actually excludes an interior distinguished zero whose critical
points all have distance at least one.  For Sendov's assertion, the scaling
of a hypothetical strict violation supplies exactly that interior setting,
including when the original distinguished zero lies on the unit circle.

\section{A short analytic tail below \texorpdfstring{$16/25$}{16/25}}
\label{sec:tail64}

This appendix supplies the tail needed for the stronger geometric bound.
It uses the two-chord components $T_0,T_1,T_2$ defined in
Section~\ref{sec:geometry} and only one base degree. Throughout this
appendix, $\eta=\ec$. It is independent of the earlier 77-row certificate.

\begin{theorem}\label{thm:tail64}
For every real $m\ge81$ and every $0\le\alpha\le17$,
\[
 F(m,\alpha)<\frac{77663}{121500}<\frac{16}{25}.
\]
\end{theorem}

\subsection{The three maxima at the base degree}

At $m=81$ each $T_j$ increases with $\alpha$ on $[0,17]$.
Here are explicit sign witnesses. Put
\[
 N=29\alpha^2+180\alpha+270,\qquad
 \mathcal B=4050(\alpha+3)^3b(81,\alpha),\qquad
 \mathcal D=4050(\alpha+3)^3(1-b(81,\alpha)).
\]
These quantities are positive. Direct logarithmic differentiation gives
\begin{align*}
 \partial_\alpha\log T_0
 &=\frac{P_0}{(\alpha+3)(81-2\alpha)N\mathcal D},\\
 \partial_\alpha\log T_1
 &=-\frac{P_1}{2(\alpha+3)(81-2\alpha)N\mathcal B},\\
 \partial_\alpha\log T_2
 &=-\frac{P_2}{2\alpha(\alpha+3)(81-2\alpha)N},
\end{align*}
where
\begin{align*}
P_0={}&\alpha^6(148662-4176\alpha)\\
 &+\alpha^3(-1198557\alpha^2+4883328\alpha+308496762)\\
 &+1928833398\alpha^2+4474050876\alpha+3603169980,\\
P_1={}&\alpha^6(111360\alpha-1901628)-12555810\alpha^5
       -1117413306\alpha^4\\
 &-14822694792\alpha^3-72036543240\alpha^2
       -150384471048\alpha-114359017320,\\
P_2={}&\alpha^2(116\alpha^2+18726\alpha-391986)
       -3235140\alpha-5117580.
\end{align*}
On $[0,17]$, the first parenthesis in $P_0$ is at least $77670$.
Its quadratic parenthesis is concave and has endpoint values
$308496762$ and $45130365$, both positive. Thus $P_0>0$.
The first parenthesis in $P_1$ is at most $-8508$, and every remaining
coefficient is negative, so $P_1<0$. The quadratic parenthesis in $P_2$
increases to $-40120$, so $P_2<0$. This proves the asserted monotonicity
for $\alpha>0$ and then at zero by continuity.

The exact endpoint data are
\[
 \eta(17)=\frac{199087}{240000},\quad
 r(17)=\frac{11711}{240000},\quad
 b(81,17)=\frac{8191349}{10800000},\quad
 K(81,17)=\frac{1060653559}{486000}.
\]
They imply the deliberately coarse rational budgets
\begin{equation}\label{eq:tail-base-budgets}
 T_0(81,17)<\frac{17}{100},\qquad
 T_1(81,17)<\frac1{200},\qquad
 T_2(81,17)<\frac{207}{1000}.
\end{equation}
For the first, substitute the displayed $r,b$ in
\[
 \frac{24(3/5)^4 r\,47^2}{81\cdot84\cdot86(1-b)^4}
 <\frac{17}{100}.
\]
For the other two, the endpoint data give
\[
 A_hK<181,\quad B_hK<295,\quad b<\frac{19}{25},\quad
 \eta<\frac{83}{100}.
\]
The following simple upper bounds verify the required powers by repeated
squaring. The $q$ and $q^2$ entries are exact; the entries from $q^4$ onward are strict upper bounds.
\[
\begin{array}{c|ccccc}
q&q^2&q^4&q^8&q^{16}&q^{32}\\ \hline
19/25&0.5776&0.334&0.112&0.0126&0.000159\\
83/100&0.6889&0.4746&0.2253&0.05077&0.002578
\end{array}
\]
Every decimal in this table denotes its exact terminating rational value.
Since $39=32+4+2+1$, direct multiplication gives
\[
181(0.000159)(0.334)(0.5776)(0.76)<\frac1{200},
\]
\[
295(0.002578)(0.4746)(0.6889)(0.83)<\frac{207}{1000}.
\]
This proves \eqref{eq:tail-base-budgets} without a large power expansion.

\subsection{Monotonicity for every larger real degree}

Fix $0<\alpha\le17$ and let $m\ge81$ be real. Write
\[
 d_0=\frac{21}{25}-\frac35\eta,\qquad
 1-b=d_0-\frac{12\alpha}{25m}.
\]
Then
\[
 \partial_m\log T_0=
 \frac3m+\frac2{m-2\alpha}-\frac1{m+3}-\frac1{m+5}
 -\frac{4d_0}{d_0m-12\alpha/25}.
\]
Since $0<d_0\le21/25$,
\[
 \frac{d_0}{d_0m-12\alpha/25}\ge\frac1{m-4\alpha/7}.
\]
The expression $2/(m-2\alpha)-4/(m-4\alpha/7)$ increases with $\alpha$,
because its derivative is
$4/(m-2\alpha)^2-(16/7)/(m-4\alpha/7)^2>0$.
Consequently
\[
 \partial_m\log T_0\le
 -\frac{N_\infty(m)}{m(m-34)(m+3)(m+5)(7m-68)}<0,
\]
where positivity of
$N_\infty(m)=7m^4-510m^3-4049m^2-35462m-104040$
is immediate from
\[
 N_\infty(81+y)=7y^4+1758y^3+147583y^2+4150618y+750186.
\]

For $K$, direct differentiation gives
\[
 \partial_m\log K
 =\frac1{m-1}+\frac1{m+1}-\frac1m+\frac2{m-2\alpha}
 <\frac1{80}+\frac2{47}.
\]
Also $b$ decreases with $m$, and
\[
 \eta\le\eta(17)<\frac56,\qquad
 b<\frac4{25}+\frac12+\frac{68}{675}<\frac45.
\]
Using $\log z\le z-1$ therefore yields
\begin{align*}
 \partial_m\log T_1
 &=\partial_m\log K+\frac12\log b
       +\frac{m-3}{2}\frac{\partial_m b}{b}
 <\frac1{80}+\frac2{47}-\frac1{10}<0,\\
 \partial_m\log T_2
 &=\partial_m\log K+\frac12\log\eta
 <\frac1{80}+\frac2{47}-\frac1{12}<0.
\end{align*}
Finally $E(m,\alpha)\le e_m=1/4+7/(12m)$ and $e_m$ decreases.
The base-degree estimates give, uniformly for $m\ge81$,
\[
 F\le e_m+T_0+T_1+T_2
 <\frac{125}{486}+\frac{17}{100}+\frac1{200}+\frac{207}{1000}
 =\frac{77663}{121500}<\frac{16}{25}.
\]
The final gap is $97/121500>0$. By continuity,
$T_j(m,0)\le T_j(81,17)$ for each component. Their strict base budgets
therefore give the same strict ceiling at $\alpha=0$, proving
Theorem~\ref{thm:tail64}.

\section{Methods and bounds for the geometric certificates}
\label{app:geometric-methods}

This appendix specifies the proof rules behind the optional geometric
results. The main analytic proof does not use this appendix or its
computer-assisted certificates.
For integer $5\le m\le50$, the finite certificates establish the curvature
partition in Theorem~\ref{thm:geometry-classification}, place $\ell(m)$
strictly in its concave part, and prove
\[
 F_\alpha(m,\ell(m))<-\frac1{10000},\qquad
 F(m,A_m)<\frac35,\qquad
 \Theta(m)>\max\{F(m,0),F(m,A_m)\}.
\]
Here $\Theta(m)=F(m,\ell)-\ell F_\alpha(m,\ell)$; this notation
distinguishes the tangent intercept from the powered contribution $S$
in the analytic majorant.
They also give $\Theta(m)\le\Theta(5)<639863/10^6<16/25$.
Together with the tangent argument in Section~\ref{sec:geometry}, these
are the finite-range hypotheses of the stronger scalar bound.
For $51\le m\le82$ the complete curvature and endpoint certificates give
\[
 F(m,\alpha)\le\max\{F(m,0),F(m,17)\}<\frac{29}{50}.
\]
The largest certified endpoint upper bound occurs at $m=51$ and is less
than $579842/10^6$. The stored certificates and reconstruction programs
are in \texttt{proof/smooth/} and \texttt{proof/high\_degrees/}.

\subsection{The origin derivative identity}

The exact formula \eqref{eq:geometry-origin-slope} can be checked without
a curvature calculation. At $\alpha=0$,
\[
 E_\alpha=-\frac{m+1}{36m},\qquad
 K_\alpha=-\frac{(m^2-1)(m+12)}{36},\qquad
 J=\frac6{(m-2)(m-1)m(m+1)}.
\]
Differentiating the kernel power and integrating the two resulting beta
integrals gives, for $m>4$,
\[
 J_\alpha(m,0)=
 \frac4{(m-4)(m-2)(m-1)m}
 +\frac{24}{(m-2)(m-1)m^2(m+1)}.
\]
All endpoint integrals converge. Substitution in
$F_\alpha=E_\alpha+K_\alpha J+KJ_\alpha$ yields
\eqref{eq:geometry-origin-slope}. This identity is distinct from the
second-derivative singularity at $m=6$ discussed below.

\subsection{Odd degrees: exact polynomial arithmetic}

When $k=(m-3)/2$ is an integer, the kernel power can be integrated exactly.
For the formula used at $m\ge9$, put
\[
Z=\alpha+3,\quad N=29\alpha^2+180\alpha+270,\quad
D=30Z^3,\quad V=\alpha N,\quad W=D-V.
\]
Then $W=\alpha^3+90\alpha^2+540\alpha+810>0$ for $\alpha\ge0$.
Writing $P_B=mD$, $Q_B=mV+2\alpha D$, and $R_B=mV$ gives
\[
B(t)=\frac{P_B(1-t)^2+Q_Bt(1-t)+R_Bt^2}{mD}
\]
and the finite beta-sum identity
\begin{equation}
J=\frac{1}{(mD)^k(2k+4)!}
\sum_{i+j+l=k}\frac{k!}{i!j!l!}P_B^iQ_B^jR_B^l
(j+2l+3)!\,(2i+j)!.\label{eq:betasum}
\end{equation}
All coefficients are rational or integer. Differentiating the resulting
rational $F$ and clearing its positive denominator reduces the curvature
question to the sign of a polynomial.

The verifier transforms that polynomial to the Bernstein basis on the capped
interval. Sign variation and exact rational subdivision exclude roots on every
complementary interval. Each accepted root bracket has variation one and
opposite endpoint signs, proving one simple root with its stated crossing
direction. The brackets and complementary intervals cover the entire domain.
This argument counts all roots; it is not a sampled sign-change search.

For $m=5,7$, direct monomial integration is shorter. There is also a concise
Descartes proof: the curvature numerator has respectively one and two changes
of coefficient sign, while its signs at $0,\ell,m/2$ are respectively
$(-,-,+)$ and $(+,-,+)$. These force the one-root and two-root patterns and
locate the anchor in the concave region. The exact coefficients, signs, and
finer Bernstein brackets are retained in the low-degree certificates.

\subsection{Even degrees: interval enclosures with remainders}

Every arithmetic endpoint is an integer multiple of $2^{-80}$. Arithmetic
operations round outward, and square roots are bounded with integer square
roots. Ordinary floating-point displays do not decide any proof sign.

On a panel of length $h$, a smooth integrand $G$ satisfies
\begin{equation}
\int_{t_0}^{t_0+h}G(t)\dd t
=\frac h6\{G(t_0)+4G(t_0+h/2)+G(t_0+h)\}
-\frac{h^5}{2880}G^{(4)}(\xi).\label{eq:simpson}
\end{equation}
The implementation encloses the full fourth derivative over the panel and
the current $\alpha$ interval, so \eqref{eq:simpson} is a rigorous integral
enclosure. Derivatives of the rational envelope are included throughout.

An inflection bracket is accepted only when $F_{\alpha\alpha}$ has opposite
strict signs at its endpoints and $F_{\alpha\alpha\alpha}$ has a strict
constant sign throughout it. Existence and uniqueness then follow from
continuity and monotonicity. On the complement, direct interval estimates or
centered Taylor estimates in $\alpha$ establish the required curvature signs.
The code verifies a complete cover, without gaps.

The reconstruction for $m=9,\ldots,50$ contains 73 simple inflections,
1,984 even-degree complementary sign intervals, and 1,386 initial odd-degree
Bernstein coefficients. The even degrees $m=10,12,14,16,18$ use direct interval
Riemann enclosures near zero to avoid artificial singularities in unsimplified
high-order remainder formulas.

\subsection{The origin argument at \texorpdfstring{$m=6$}{m=6}}

With primes denoting $\alpha$ derivatives, for $\alpha>0$ one has
\begin{align}
F''=E''+\int_0^1t^3\bigl[&K''B^p+2K'pB^{p-1}B'
 +Kp(p-1)B^{p-2}(B')^2\nonumber\\
 &+KpB^{p-1}B''\bigr]\dd t.\label{eq:second}
\end{align}
At $m=6$, $p-2=-1/2$, and the third term causes positive logarithmic divergence
of $F''$ as $\alpha\downarrow0$. A finite upper enclosure of $F''$ at zero
would be invalid.

For $0\le\alpha\le1/20$, partition the $t$ interval into 256 rational panels.
On each panel, a positive upper bound $U_B$ gives
$B^{-1/2}\ge U_B^{-1/2}$ when $\alpha>0$. Its coefficient in
\eqref{eq:second} is nonnegative, so replacing only this factor gives a
valid one-sided lower bound. Interval arithmetic on the other terms proves
\[
F''(6,\alpha)\ge
\frac{28875847078188054828007}{302231454903657293676544}
>\frac9{100},\qquad 0<\alpha\le\frac1{20}.
\]
Only the lower endpoint of this surrogate expression bounds the true
curvature; its upper endpoint is not used. Continuity of $F$ at zero extends
convexity to the closed interval.

At $m=8$ all powers in \eqref{eq:second} are nonnegative, and an ordinary
interval Riemann enclosure gives
\[
F''(8,\alpha)\ge
\frac{64831826394827463289543}{1208925819614629174706176}
>\frac1{20},\qquad 0\le\alpha\le\frac1{20}.
\]
The third derivative at zero is not used. Beyond $1/20$, positive kernel
bounds permit the usual Simpson and root-isolation machinery. The complete
even low-degree covers have 27 intervals at $m=6$ and 39 at $m=8$.

\subsection{Point values and the exact limiting row}

All evaluations at $\ell(m)$ and $A_m$ have strictly positive $\alpha$.
Consequently $B(t)>0$ on $[0,1]$, and the 64-panel Simpson calculation with
fourth-derivative enclosures is valid for every $m=5,\ldots,50$, including
the small degrees. The certificate retains exact dyadic endpoints for
$F(\ell)$, $F_\alpha(\ell)$, $F(A_m)$, and $\Theta(m)$.

At the worst row, $m=5$, the integral is elementary:
\[
J_5=\frac1{60}+\frac{\ec}{5}+\frac{\alpha}{75},\qquad
E_5=\frac{r}{2(1-\ec)}+\frac{3r}{10}+\frac1{10},\qquad
K_5=\frac65r(5-2\alpha)^2.
\]
Exact substitution and differentiation at $\ell(5)=87/200$ yield
\[
\Theta(5)=\frac{117697989197625021749025195674259761143733}
{183942742442208379835690455341276881250000}.
\]
This exact value lies between $0.639862098579$ and $0.639862098580$.
All other certified upper bounds are smaller than this exact $\Theta(5)$, proving
the optional sharper uniform ceiling in Corollary~\ref{cor:global-geometry64}.

\section*{Acknowledgments and use of language models}
The author acknowledges Mazur's original proof architecture, Tao's
exposition of the polynomial reduction, and Zhang's weighted-kernel
estimates and small-parameter argument used in Section~\ref{sec:analytic-uniform}. ChatGPT was used to explore
proof strategies, suggest inequalities and changes of variables,
develop symbolic and exact-arithmetic cross-checks, and assist with
exposition. The uniform analytic completion is the written argument; the geometric
results have their separately identified computational certificates, and
the three-envelope continuation is numerical. Responsibility for the mathematical statements and their
presentation remains with the author.

\clearpage
\begingroup\small\raggedright

\endgroup
\end{document}